\documentclass[11pt]{amsart}  
\usepackage{amssymb,amsmath,amsthm,amscd}
\usepackage{graphicx}
\usepackage{tikz-cd}
\usepackage{lscape}
\newcommand{\ep}[0]{\epsilon}

\newcommand{\mat}[4]{\left(\begin{array}{cc}#1&#2\\#3&#4\end{array}\right)}

\newcommand{\ov}[1]{\frac{1}{#1}}
\newcommand{\f}[1]{\mathbb{#1}}

\newcommand{\ar}[1]{\mathbb{R}^{#1}}

\newcommand{\mods}[1]{\mathfrak{M}_{\text{#1}>0}}
\newcommand{\modse}[1]{\mathfrak{M}_{\text{#1}\geq0}}

\newcommand{\zco}[2]{H^{#2}(#1,\f{Z})}
\newcommand{\twoco}[2]{H^{#2}(#1,\f{Z}_2)}

\newcommand{\dk}[1]{\text{dim}(\text{ker}(#1))}
\newcommand{\bs}{\backslash}
\newcommand{\du}[2]{{\underline{#1},\underline{#2}}}

\numberwithin{equation}{section}

\newtheoremstyle{fancy1}{10pt}{10pt}{\itshape}{12pt}{\textsc\bgroup}{.\egroup}{8pt}{
}
\newtheoremstyle{fancy2}{10pt}{10pt}{}{12pt}{\itshape}{.}{8pt}{ }

\theoremstyle{fancy1}

\newtheorem{lem}[equation]{Lemma}

\newtheorem{thm}[equation]{Theorem}
\newtheorem*{thm*}{Theorem}

\newtheorem{main}{Theorem}
\newtheorem*{main*}{Theorem}

\newtheorem*{cor*}{Corollary}
\newtheorem*{prop*}{Proposition}

\newtheorem*{problem*}{Problem}

\newtheorem{maincor}[main]{Corollary}

\theoremstyle{fancy2}
\newtheorem{definition}[equation]{Definition}

\newtheorem*{rems*}{Remarks}

\newtheorem*{rem*}{Remark}

\newtheorem*{example*}{Example}

\newcommand{\cref}[1]{Corollary~\ref{#1}}

\newcommand{\lref}[1]{Lemma~\ref{#1}}

\newcommand{\tref}[1]{Theorem~\ref{#1}}
\newcommand{\sref}[1]{Section~\ref{#1}}

\newcommand{\C}{{\mathbb{C}}}
\newcommand{\R}{{\mathbb{R}}}
\newcommand{\Z}{{\mathbb{Z}}}
 \newcommand{\Q}{{\mathbb{Q}}}

\newcommand{\SO}{\ensuremath{\operatorname{SO}}}

\renewcommand{\O}{\ensuremath{\operatorname{O}}}

\newcommand{\Spin}{\ensuremath{\operatorname{Spin}}}
\newcommand{\Pin}{\ensuremath{\operatorname{Pin}}}

\newcommand{\K}{\ensuremath{\operatorname{K}}}

\def\con#1=#2(#3){#1 \equiv #2 \bmod{#3}}

\newcommand{\diag}{\ensuremath{\operatorname{diag}}}

\renewcommand{\sec}{\ensuremath{\operatorname{sec}}}

\begin{document}

\title{ Moduli Spaces of Nonnegatively Curved metrics on  Exotic Spheres}

\author{McFeely Jackson Goodman}
\address{University of California, Berkeley}
\email{mjgoodman@berkeley.edu}
\thanks{This research was partially supported by National Science Foundation grant  DMS-2001985}

\begin{abstract}
We show that the moduli space of nonnegatively curved metrics on each member of a large class of 2-connected 7-manifolds, including each smooth manifold homeomorphic to  \(S^7\),  has infinitely many connected components.   The components are distinguished using the Kreck-Stolz \(s\)-invariant computed for metrics constructed by Goette, Kerin and Shankar.  The invariant is computed by extending each metric to the total space of an orbifold disc bundle and applying generalizations of the Atiyah-Patodi-Singer index theorem for orbifolds with boundary.  We develop methods for computing characteristic classes and integrals of characteristic forms appearing in index theorems for orbifolds, in particular orbifolds constructed using Lie group actions of cohomogeneity one.     
\end{abstract}

\maketitle
An exotic \(n\)-sphere is a smooth manifold which is homeomorphic but not diffeomorphic to \(S^n\).  
When \(n=7\) there are 14  diffeomorphism types of exotic spheres (ignoring orientation). Ten,  the  Milnor spheres, are represented by total spaces of \(S^3\) bundles over \(S^4\) and four are not.  Given that \(S^n\) is the prototypical example of positive curvature, one asks whether any exotic spheres admit Riemannian metrics satisfying similar curvature conditions. After Gromoll and Meyer showed that one Milnor sphere admits a nonnegatively curved biquotient metric, Grove and Ziller  used cohomogeneity one actions of Lie groups to construct metrics of nonnegative curvature on all of the Milnor spheres \cite{GZ}. 

 More recently, for each pair of triplets of integers \(\underline{a}=(a_1,a_2,a_3)\) and \(\underline{b}=(b_1,b_2,b_3)\)  
 satisfying certain gcd conditions (see \eqref{cond1} and \eqref{cond2}), Goette, Kerin, and Shankar \cite{GKS} generalized the Grove-Ziller construction to identify a 2-connected 7-manifold \(M_{\underline{a},\underline{b}}\) admitting a metric of nonnegative sectional curvature.  Each exotic 7-sphere (including those which are not Milnor spheres) is diffeomorphic to some \(M_\du{a}{b}\).
 
 In the present paper we study the range of nonnegatively curved metrics on each \(M_\du{a}{b}\), and each exotic sphere. We consider the moduli space 

 \[\modse{sec}(M)=\{\text{Riemannian metrics on }M\text{ with sec}\geq0\}/\text{Diff}(M).\]    
 The diffeomorphism group acts by pulling back metrics.  We first prove  

 \begin{main}\label{odd}
 	If \(a_1\) and \(b_1\) are relatively prime and \(\zco{M_\du{a}{b}}{4}\) has odd order then  \(\modse{sec}(M_{\underline{a},\underline{b}})\) has infinitely many connected components.
 \end{main}

Dessai \cite{D} and the author \cite{G} independently showed that for each Milnor sphere \(\Sigma\), including \(S^7\), Grove-Ziller metrics represent infinitely many path components of \(\modse{sec}(\Sigma)\).  The results of \cite{GZ,GKS} show that each exotic spheres is diffeomorphic to some \(M_{\underline{a},\underline{b}}\) with \(b_1=1\). \tref{odd} then implies

 	\begin{maincor}\label{thm}
 	If \(\Sigma\) is a smooth manifold homeomorphic to \(S^7\)  then \(\modse{sec}(\Sigma)\) has infinitely many connected components.  
 \end{maincor}


	For those \(M_{\underline{a},\underline{b}}\) which do not satisfy the conditions of \tref{odd}, the diffeomorphism type is given up to finite ambiguity, leading to the following existence result, which nonetheless describes novel topological types with the same moduli space property.
	
	\begin{main}\label{thmB}
	For each \((\underline{a},\underline{b})\) such that \(\zco{M_\du{a}{b}}{4}\) is finite there exists \((\underline{a'},\underline{b'})\) such that \(\modse{sec}(M_{\underline{a'},\underline{b'}})\) has infinitely many connected components,  \(\mu(M_{\underline{a},\underline{b}})=\mu(M_{\underline{a'},\underline{b'}})\) , and there exists an isomorphism  \(H^4(M_{\underline{a},\underline{b}},\f{Z})\to H^4(M_{\underline{a'},\underline{b'}},\f{Z})\) preserving the linking form and the first Pontryagin class.
		\end{main}

Here \(\mu\) is the Eels-Kuiper diffeomorphism invariant and the linking form is a homotopy invariant (see \sref{exotic}).  In the proof of \cite[Thm. A]{GKSl}, (see Theorem 4.3 and Corollary 4.4 in that paper) the authors demonstrate that infinitely many diffeomorphism types of \(M_{\underline{a},\underline{b}}\) with finite \(H^4\) have the cohomology ring but not  the homotopy type of an \(S^3\) bundle over \(S^4\).  The obstruction is the linking form.  It follows from \tref{thmB} that such diffeomorphism types are  represented by manifolds with the property that \(\modse{sec}\) has infinitely many components.

We distinguish the  components of  \(\modse{sec}\)  using the \(s\) invariant defined by Kreck and Stolz  \cite{KS}.  For a Riemannian metric \(g\) with positive scalar curvature on a  spin manifold \(M^{4k-1}\) with vanishing real Pontryagin classes, \(s(M,g) \in\f{Q}\) is a linear combination of spectral invariants of geometric differential operators and integrals of differential forms related to curvature. \(s(M,g)\) is invariant under pullbacks by spin diffeomorphisms and continuous modifications of \(g\) which preserve positive scalar curvature.  

Strengthening a result of Carr \cite{C}, Kreck and Stolz showed that for a simply connected \(4k-1\)-manifold on which \(s\) can be defined, if \(k\geq2\) then \(\mods{scal}(M)\) is either empty or has infinitely many path components.  Also using the \(s\) invariant, Wraith showed that if \(\Sigma\)  is homeomorphic to \(S^{4k-1}\) and bounds a parallelisable manifold then \(\mods{Ric}(\Sigma)\) has infinitely many path components \cite{W} .  We note that \(S^7\) and all of the exotic 7-spheres bound parallelisable manifolds (see \cite{KM}) so Wraith's theorem applies to  the manifolds in \cref{thm}.   Using a related approach Dessai and Gonz\'{a}lez-\'{A}lvaro \cite{DG} and Wermelinger \cite{We} showed that \(\modse{sec}\) has infinitely many components for all manifolds homeomorphic to \(\ar{}P^5\) and certain manifolds homeomorphic to \(\ar{}P^7\).  
 
 We compute the \(s\) invariant by describing each \(M_\du{a}{b}\) as the total space of an orbifold \(S^3\) bundle, and therefore the boundary of an orbifold disc bundle.  We then apply the index theorem for orbifolds with boundary to write \(s\) in terms of  orbifold topological data and the index of the Dirac operator on the orbifold disc bundle.  
 The primary difficulty in proving Theorems \ref{odd} and \ref{thmB} involves evaluating the contributions to the index theorem of the singular parts of the orbifold bundles. To the authors knowledge, this is the first time orbifold methods have been used to compute moduli space invariants for nonnegative curvature or related conditions.  We believe the techniques developed in this paper will be of general utility in future work computing index theoretic quantities for orbifolds, particularly orbifolds described using cohomogeneity one Lie group actions.  Such quantities can in turn be used to compute  moduli space invariants and diffeomorphism invariants.

This paper is organized as follows. In \sref{ks} we generalize the formula in \cite{G}, following \cite{KS}, to evaluate the s invariant of the total space of an orbifold \(S^3\) bundle.  Kawasaki \cite{K} and Farsi \cite{F} generalized the Atiyah-Patodi-Singer index theorem \cite{APS} to the orbifold case.   



The index theorems for the Dirac and signature operators involve  integrals of orbifold versions of the \(\hat{A}\) and \(L\) forms.  The forms are integrated over the inertia orbifold, an auxiliary orbifold which encodes the singular strata.  In \sref{quotientinertia} we give a global description of the inertia orbifold of a quotient of a manifold by an almost free Lie group action, and global formulas for the integrals which appear in orbifold index theorems (\lref{lgm}).   In \sref{coho1inertia} we apply these formulas to the specific case of certain orbifolds, including orbifold disc bundles, obtained from manifolds admitting cohomogeneity one Lie group actions.  The singular parts of the integrals can then be computed using the data of the cohomogeniety one action (\lref{coho1}).  

In \sref{mfds}, 
we describe the 7-manifolds \(M_{\underline{a},\underline{b}}\) studied in \cite{GKS}, which are quotients of 10-manifolds \(P_\du{a}{b}\) admitting cohomogeneity actions of \(S^3\times S^3\times S^3\) and metrics of nonnegative curvature described in \cite{GZ}.  
 By means of Riemannian submersion  metrics on \(P_{\du{a}{b}}\)  induce nonnegatively curved metrics on the total space of the orbifold \(S^3\) bundles, which in turn extend to metrics of positive scalar curvature on the associated orbifold \(D^4\) bundles.  We use the tools of the previous sections to compute the \(s\) invariant for \(M_{\underline{a},\underline{b}}\).  The contribution of the index of the Dirac operator on the orbifold disc bundle vanishes by Lichnerowicz' theorem. 



In \sref{diffinv} we use the diffeomorphism classification, due to Crowley, of 2-connected rational cohomology 7-spheres,  to find infinite families of manifolds \(M_\du{a}{b}\) which are all diffeomorphic, but such that the \(s\) invariant computed for the nonnegatively curved metrics take on infinitely many values.  We conclude that those metrics, pulled back to any \(M_\du{a}{b}\) in the family, will represent infinitely many distinct components of \(\modse{sec}.\)  The diffeomorphism invariants  are the Eels-Kuiper invariant \(\mu,\) which can be recovered from \(s\) (reproducing a computation in \cite{GKS}) along with the linking form, which is computed in \cite{GKSl}, and the first Pontryagin class, which we compute in \sref{pontryagin}, using methods from orbifold topology, in the case where \(a_1\) and \(b_1\) are relatively prime.  If \(\zco{M_\du{a}{b}}{4}\) has odd order, these invariant determine the diffeomorphism type completely, allowing use to prove \tref{odd} and \cref{thm}.  In the general case those invariants determine diffeomorphism type up to finite ambiguity, implying \tref{thmB}.    \sref{signs} contains a detailed definition of a sign convention mentioned used in \sref{index} and \sref{quotientinertia}.

 A complete diffeomorphism classification of the manifolds \(M_\du{a}{b}\) will appear in a forthcoming work by Goette, Kerin and Shankar.

I would like to thank Wolfgang Ziller, Martin Kerin, and Ravi Shankar for invaluable conversations and guidance.



\section{Preliminaries}\label{pre}

\subsection{Orbifolds}\label{orbifolds}

We briefly summarize the properties of orbifolds which we will use.  For details, see \cite{ALR}, \cite{GKS}, \cite{K}, \cite{K2}, \cite{S}. An \emph{ orbifold chart} for a topological space \(X\) centered at \(p\in X\) consists of an open set \(V_p\subset X\), an open disc \(U_p\) in \(\ar{n}\) centered at the origin, a finite subgroup \(\Gamma_p\) of \(\O(n)\) and a \(\Gamma_p\)-invariant map \(\phi_p:U_p\to V_p\), \(\phi_p(0)=p,\)  which induces a homeomorphism between \(\Gamma_p\backslash U_p\) and \(V_p.\)  We will often reuse the notation \(\phi_p\) for that homeomorphism.  \(\Gamma_p\) is called the isotropy group of  \(p\).    An \(n\)-\emph{orbifold} has orbifold charts centered at each point, satisfying the following compatibility condition:  if \(\phi_p:U_p\to V_p\) and \(\phi_{p'}:U_{p'}\to V_{p'}\) are two orbifold charts, for each \(q\in V_p\cap V_{p'}\) there exists an open set \(V_q\subset V_p\cap V_{p'}\) containing \(q\), an orbifold chart \(\phi_q:U_q\to V_q\) centered at \(q\), and open embeddings \(\psi\), \(\psi'\) such that the following diagram commutes:

\[\begin{tikzcd}U_{p'}\arrow{d}{\phi_{p'}}&U_q\arrow{r}{\psi}\arrow{l}{\psi'}\arrow{d}{\phi_q}&U_p\arrow{d}{\phi_p}\\V_{p'}&\arrow{l}V_q\arrow{r}&V_p\end{tikzcd}\]

This compatibility condition implies that \(\psi\) is equivariant with respect to an injective homomorphism \(\theta:\Gamma_q\to \Gamma_p.\)  It follows that two orbifold charts centered at the same point will have isomorphic isotropy groups.

The orbifolds in this paper will be constructed using Lie group actions on smooth manifolds.  We first fix some notation.  Let \(G\) be a Lie group acting on the left of a space \(X.\)  For \(p\in X\), \([p]\in G\backslash X \) will denote the orbit of \(p\) (\([p]\in X/ G\) if the action was on the right).   If \(G\) acts on \(X\) and \(Y\) on the left, \(G\bs(X\times Y)\) denotes the quotient of the diagonal action \(g(x,y)=(gx,gy)\).    If \(G\) acts on \(X\) on the right and on \(Y\) on the left, we define a left action of \(G\) on \(X\times Y\) such that \(g\in G\) acts by
\[g\cdot(x,y)=(xg^{-1},gy).\]  
The quotient of \(X\times Y\) by this action will be denoted \(X\times _{G}Y.\)     If \(G\) acts on \(X\) on the left and \(H\) acts on \(X\) on the right then \(G\bs X/H\) denotes the quotient of \(X\) by the left action of \(G\times H\) described by
\[(g,h)\cdot x=gxh^{-1}.\]

\(G_x\subset G\) will denote the stabilizer of a point \(x\in X\).  We say a compact  Lie group \(G^{n-k}\) act almost freely on the left of a smooth manifold \(M^n\) if the stabilizer of each point in \(M\) is a finite subgroup of \(G\). It follows that \(X^k=G\backslash M\) is an orbifold.    The orthonormal frame bundle (see \sref{orbibundles}) of an orbifold is always a manifold so any orbifold can be described in this way.   

Lie groups will be assumed to have fixed orientations.  If \(M\) is oriented, and a connected Lie group \(G\) acts freely or almost freely on \(M,\) the vertical space of the action at a point \(p\in M\) can be identified with the Lie algebra \(\mathfrak{g}\) of \(G\) by means of the action fields.  Then the horizontal space \(H_p\) is oriented such that \(H_p\oplus\mathfrak{g}\) induces the orientation on \(T_pM.\)  The projection \(M\to G\bs M\) identifies \(H_p\) with \(T_{[p]}(G\bs M)\) (see \sref{orbibundles} for the orbifold case), orienting \(G\bs M.\)

We will use the orbifold structure of such a quotient in detail in \sref{quotientinertia}, and recall it here.  Let \(\pi:M\to X\) be the quotient map.  The slice theorem states that for each \(p\in M\)  there exists an embedded \(k\)-disc \(U_p\), called a \emph{ slice neighborhood centered at \(p\)} which contains \(p\) and is preserved by the stabilizer group \(G_p\) such that the map 
\[G\times_{G_p}U_p\to M\]
\[[g,p]\mapsto gp\]
is an equivariant diffeomorphism onto a tubular neighborhood of the \(G\) orbit of \(p.\)  It follows, as we will use often, that if \(x,y\in U_p\) and \(gx=y\) for \(g\in G\) then \(g\in G_p\).   \(\pi\) induces a homeomorphism between \(G_p\backslash U_p \) and the quotient of that tubular neighborhood by \(G\), which is \(\pi(U_p)\).  So \(U_p\)  is an orbifold chart for \(X\) centered at \([p]\)  with isotropy group \(G_p.\)  

The open embedding \(\psi\) described in the compatibility condition above are defined as follows.  If  \([q]\in\pi(U_p)\) then there exists \(g\in G\) such that \(gq\in U_p.\)  Let \(U_q\) be a slice neighborhood centered at \(q,\) shrunk if necessary such that \(\pi(U_q)\subset\pi(U_p)\), then there is a continuous map \(f:U_q\to G\) such that \(f(q)=g\) and \(f(x)x\in U_p\) for all \(x\in U_q\) (The diffeomorphisms from the slice theorem induce an embedding \(U_q\to G\times_{G_q}U_q\to G\times_{G_p}U_p\), which has a lift \(U_q\to G\times U_p\) such that \(q\) maps to \((g^{-1},gq)\).  \(f\) is the composition of that lift with projection onto \(G\) and the group inverse map.)  Then we define \(\psi:U_q\to U_p\) by \(\phi(x)=f(x)x,\) so \(\pi\circ \phi=\pi.\)     

We illustrate the injective homomorphism \(\theta:G_q\to G_p\) with respect to which \(\phi\) is equivariant.  For each \(\gamma\in G_q,\) \(f(\gamma x)\gamma f(x)^{-1}\) sends \(f(x)x\) to \(f(\gamma x)\gamma x.\)  Since both are in \(U_p,\) \(f(\gamma x)\gamma f(x)^{-1}\in G_p\) for all \(x\in U_q.\)  Since \(f\) is continuous and \(G_p\) is finite it follows that \(f(\gamma x)\gamma f(x)^{-1}\) is independent of \(x\) and equal to its value at \(q,\) which is \(g\gamma g^{-1}.\)  It follows that \(\phi(\gamma x)=g\gamma g^{-1}\phi(x)\) for all \(x\in U_q\) and \(\gamma\in G_q.\) Thus \(\theta(\gamma)= g\gamma g^{-1}.\)


\subsection{Orbifold Bundles}\label{orbibundles}
An orbifold bundle with fiber \(F\) is a smooth orbifold map \(\pi:Y\to X\) such that for each \(p\in X\) there is an orbifold chart \(\phi_p: U_p\to V_p\) with isotropy group \(\Gamma_p\), an action of \(\Gamma_p\) on \(U_p\times F\) such that projection onto \(U_p\) is equivariant and a \(\Gamma_p\)-invariant map \(\psi_p\) making the following diagram commute
\[\begin{tikzcd}
U_p\times F\arrow{d}{\psi_p}\arrow{r}&U_p\arrow{d}{\phi_p}\\\pi^{-1}(V_p)\arrow{r}{\pi}&V_p
\end{tikzcd}\]

  \noindent Further, \(\psi_p\) must induce a homeomorphism between \(\Gamma_p\backslash (U_p\times F)\) and \(\pi^{-1}(V_p)\).  Two such charts obey a compatibility condition similar to the one given above, described by the following commuting diagram, where \(\psi,\psi'\) are open embeddings: 

  \[\begin{tikzcd}U_{p'}\times F\arrow{d}{\psi_{p'}}&\arrow{l}{\psi'}U_q\times F\arrow{r}{\psi}\arrow{d}{\psi_q}&U_p\times F\arrow{d}{\psi_p}\\\pi^{-1}(V_{p'})&\arrow{l}\pi^{-1}(V_q)\arrow{r}&\pi^{-1}(V_p)\end{tikzcd}\]

  \noindent A smooth section of \(\pi\) is a map \(s:X\to Y\) such that \(\pi\circ s= \) id.  If follows that in each trivializing chart \(s\) will be covered by an equivariant section of the bundle \(U_p\times F\to U_p.\)
 
 We say \(\pi:Y\to X\) is an orbifold vector bundle if \(F\) is a vector space and \(\Gamma_p\), as well as the transition functions defined by the compatibility condition,  act linearly on \(F\).   If \(F\) is a Lie group, and \(\Gamma_p\) and the transition functions  act on the left as elements of \(F,\) then \(\pi:X\to Y\) is an orbifold principal \(F\) bundle. In that case, if \(F\) acts on another space \(Z\), we can form an associated bundle over \(X\) with fiber \(Z\) in the usual way.   Note that if \(G\) acts almost freely on a manifold \(M,\) the slice theorem implies that \(M\to G\bs M\) is an orbifold principal \(G\) bundle.  
 
 \(\Gamma_p\) acts naturally on the tangent bundle and tensor bundles of each \(U_p,\) so we can define the tangent bundle \(TX\) and tensor bundles of \(X.\)  We can then define a Riemannian metric and the corresponding curvature tensors.  Locally they will correspond to the data of a \(\Gamma_p\)-invariant Riemannian metric on \(U_p.\)   A connection on an orbifold vector bundle is defined by a consistent choice of \(\Gamma_p\)-invariant connections on the orbifold charts \(U_p.\)  Given a Riemannian metric and an orientation, the oriented orthonormal frame bundle  \(\SO(TX)\)  of the tangent bundle  is an orbifold principal \(\SO(n)\)  bundle over \(X\). An orbifold spin structure for \(X\) is an orbifold  principal \(\Spin(n)\) bundle \(\Spin(TX)\) over \(X\) and an orbifold bundle map \(\rho:\Spin(TX)\to\SO(TX)\) which is covered in each trivializing orbifold chart by the double cover \(\Spin(n)\to \SO(n)\) on the fibers.

 \subsection{Differential forms and characteristic classes}\label{diffforms}
 The exterior derivative \(d\) can be defined on differential forms in \(\Omega(X)\), as well as integration on top degree forms using the local definition
 \[\int_{V_p}\alpha=\ov{|\Gamma_p|}\int_{U_p}\alpha_p\]
 where \(\alpha_p\) is the \(\Gamma_p\)-invariant form covering \(\alpha\) in an orbifold chart \(\phi_p:U_p\to V_p\) with isotropy group \(\Gamma_p\).  Stokes' theorem holds and the de Rahm theorem yields an isomorphism from the cohomology of \((\Omega(X),d)\) to the singular cohomology \(H^*(X,\ar{})\) of the topological space \(X\).  Using that identification we denote the cohomology class of a closed form \(\alpha\) by \([\alpha]\in H^*(X,\ar{})\). Poincare duality holds with real coefficients; in particular, for a closed and oriented \(n\)-orbifold, integration provides an isomorphism \(H^n(X,\ar{})\cong\ar{}\) (the proof in \cite{BT} for instance will apply just as in the manifold case).  The Thom isomorphism in de Rahm cohomology holds.  That is,  let \(\pi:E^{n+k}\to X^n \) be an oriented rank \(k\) orbifold vector bundle oriented such that the orientations on \(E\) and \(B\), and the standard orientation on \(\ar{k}\), agree via the identification \(T_pE\cong T_{\pi(p)}B\oplus \ar{k}.\)  Then there is a closed compactly supported \(k\)-form \(\Phi\) on \(E,\) called the Thom form, such that for all \(\alpha\in\Omega^n(X)\)
 \begin{equation}\label{iof}\int_E\pi^*(\alpha)\wedge \Phi=\int_X\alpha\end{equation}
and the map 
\[H^*(X,\ar{})\to H^*_{\text{c}}(E,\ar{})\]  
\[[\beta]\mapsto [\pi^*(\beta)\wedge\Phi]\]
is an isomorphism. Here \(H^*_{\text{c}}(E,\ar{})\) is compactly supported cohomology. The theorem can also be stated with the vector bundle \(E\) replaced by the \(k\)-disc bundle \(Z\subset E\). In that case \(\Phi\) is zero in a neighborhood of \(\partial Z\) and    \(H^*_{\text{c}}(Z,\ar{})\) is  replaced by \(H_c^*(Z,\partial Z;\ar{}).\)
 
 Given a connection \(\nabla\) on an orbifold vector bundle \(E\to X\), we can define Chern-Weil forms in \(\Omega(X)\) in terms of the curvature of the connection, including Pontryagin forms \(p_i(\nabla)\) for an arbitrary vector bundle and the Euler class \(e(\nabla)\) for an oriented vector bundle.  If \(\nabla\) is the Levi-Civita connection on \(TX\) corresponding to a Riemannian metric \(g\) on \(X\) we will replace \(\nabla\) with \(g\) in the characteristic form notation.  Characteristic forms defined for two connections on the same vector bundle will differ by an exact form, and so characteristic forms define characteristic classes in \(H^*(X,\ar{})\) which are independent of the choice of connection.  We will denote the Pontryagin and Euler classes by \(p_i(E)\) and \(e(E)\) respectively.  For a closed orbifold \(X\), integrals of characteristic forms depend only on the cohomology class, and characteristic numbers are well defined independent of connections.  We therefore suppress the connection and use the class in the notion, e.g. \(\int_Xp_i(E)\).  We recall that if \(\Phi\in \Omega(Y)\) is a Thom form for an oriented vector bundle \(\pi:Y\to X\) and \(e\in\Omega(X)\) is an Euler form, then \([\Phi]=\pi^*e(E)\in H^*(E,\ar{}).\)
 
 \subsection{Orbifold Characteristic classes}\label{orbiclasses}
 
 We briefly describe another construction for characteristic classes for orbifold principal bundles, which we will use in \sref{pontryagin}.  For every orbifold \(X\) there is a natural topological space \(BX\), defined up to homotopy equivalence, known as the classifying space, and a rational homotopy equivalence \(BX\to X\). 
 See \cite{ALR} for an abstract definition.  If \(X\) is expressed as the quotient of a manifold \(M\) by  an almost free Lie group action by \(G\), and \(EG\) is the total space of the universal \(G\)-bundle, (i.e. \(EG\) is a contractible space on which \(G\) acts freely),  \(BX\) is homotopy equivalent to \(G\bs (M\times E)\) (the ``Borel construction" or ``homotopy replacement" ) with the projection to \(G\bs M\).          
 
 If \(\pi:Y\to X\) is an orbifold principal \(H-\)bundle for a Lie group \(H\), then \(BY\to BX\) is a principal \(H\) bundle in the standard sense; that is, \(H\) acts freely on \(BY\) with quotient \(BX.\)  Thus the bundle has characteristic classes in \(H^*(BX)\) pulled back from \(H^*(BH)\) in the standard way.  For instance, if \(H=\SO(n)\) we can define orbifold Pontryagin classes \(p_i^{\text{orb}}\in\zco{BX}{*}\) or Euler class  \(e^{\text{orb}}\) if \(n\) is even.  
 
 If \(X\) is in fact a manifold then \(BX\) is homotopy equivalent to \(X\) (indeed, by the above description we can take \(M=X\), \(G=\{1\}\), and \(E\) a point.) Then orbifold characteristic classes and standard characteristic classes coincide.  Otherwise, 
 if \(Y\) is the frame bundle of an orbifold vector bundle \(E\to X\), the isomorphism \(H^*(X,\R)\to H^*(BX,\R)\) will map \(p_i(E)\) and \(e(E)\), as defined with Chern-Weil theory in \sref{diffforms}, to the image of \(p_i^\text{orb}\) and \(e^\text{orb}\) in \(H^*(BX,\R)\).  In this way the definitions are compatible.  This fact can be confirmed by choosing a large enough finite dimensional subset \(E_kG\) of \(EG\) such that the maps
 \[BX\simeq G\bs (M\times EG)\leftarrow G\bs (M\times E_kG)\to X=G\bs M \]
 induce isomorphisms on real cohomology  in suitable degrees and \(G\bs(M\times E_kG)\) is a manifold, where the  definitions all coincide.

  \subsection{Dirac Operators}\label{dirac}
 
 Given a Riemannian metric \(g\) and a Spin structure on an orbifold \(X\) we can define a Spinor bundle \(S\) over \(X\) and Clifford multiplication on that bundle.   The Dirac operator \(D_g\) can be defined with the usual formula on orbifold charts.  If the dimension of \(X\) is even, the spinor bundle has a decomposition \(S=S^+\otimes S^-\) and \(D_g:S^{\pm}\to S^{\mp}\).  
 
 If \(X\) is closed, \(D_g\) is a Fredholm operator, and in even dimensions the index of the Dirac operator refers to the index of \(D_g:{S^+}\to S^-,\) that is
 \[\text{ind}(D_g)=\dk{D_g|_{S^+}}-\dk{D_g|_{S^-}}.\]   
 The Lichnerowicz formula holds for \(D_g^2\) .  Integrating \(\left<D_g^2\phi,\phi\right>\) over \(X\), where \(\phi\) is a section of \(S\), we conclude just as in the manifold case that \(D_g\) has a trivial kernel if scal\((g)>0.\)  
 
 Suppose \(X\) is a compact orbifold with boundary and suppose \(\partial X\) is a manifold. Let \(h\) be a Riemannian metric on \(X\) such that there exists a collar neighborhood of \(\partial X\), diffeomorphic to \(\partial X\times I\), on which \(h\) is isometric to a product metric.  We call such a metric {\emph {product-like}.}  We can impose Atiyah-Patodi-Singer boundary conditions on sections of \(S\) such that \(D_g\) is again Fredholm operator, see \cite{APS}, \cite{F}.  Near the boundary \(X\) is a manifold and the usual arguments on the boundary  apply unchanged to prove that scal\((h)>0\) implies ker\((D_h)\) is trivial    (see for instance the argument before Theorem 3.9 in \cite{APS2} or  Lemma 1.2 in \cite{BG}).

 \subsection{Inertia Orbifold}
 In order to compute the index of a geometric operator on an orbifold \(X\), we need to define an auxiliary orbifold called the inertia orbifold \(\Lambda X\).  The connected components of \(\Lambda X\) encode the singular strata of  \(X\).  As a set, the inertia orbifold is 
 \[\Lambda X=\{(p,[g])|p\in X, [g]\text{ a conjugacy class of }\Gamma_p\}.\]
 
 We define orbifold charts for \(\Lambda X\) as follows.  Let \( U_p\xrightarrow{\phi_p}V_p\subset X\) be an orbifold chart for \(X\) centered at \(p\) with isotropy group \(\Gamma_p\).  For \(g\in \Gamma_p\) let \(U_p^g\subset U_p\) be the fixed point set of \(g\) and \(Z_{\Gamma_p}(g)\subset \Gamma_p\) the centralizer of \(g.\)  Then \(Z_{\Gamma_p}(g)\) acts on \(U_p^g\) and we define an orbifold chart 
 \[Z_{G_p}\backslash U^g_p\to \Lambda X\]
 \[[x]\mapsto (\phi_p(x),[g])\in\Lambda X\]
where \([g]\) is the conjugacy class of \(g\) in \(\Gamma_{\phi_p(x)}\).  We are using the identification 
\[\Gamma_{\phi_p(x)}\cong\{h\in \Gamma_p|hx=x\}\]
which follows from the compatibility condition for orbifold charts.  Given an open embedding \(\psi:U_q\to U_p\) as described in the compatibility conditions for the charts of \(X\) in \sref{orbifolds}, and the induced injective homomorphism \(\theta:\Gamma_q\to \Gamma_p,\) there is an open embedding \(\psi:U_q^h\to U_p^{\theta(h)}\)  for each \(h\in \Gamma_q\) satisfying the compatibility condition for the charts of \(\Lambda X.\) 
 
 We note three things.  First, different connected components of \(\Lambda X\) can have different dimensions.   Second, the union of connected components
 \[\{(p,[1])|p\in X\}\subset\Lambda X \]
 is simply  \(X,\) and we will reuse that notation for the subset.  We label the union of the remaining connected components, which correspond to nontrivial elements of isotropy groups, \(\Lambda_sX=\Lambda X\backslash X\).
   Finally, \(Z_{\Gamma_p}(g)\) may not act effectively on \(U_p^g,\) and so by strict definition a quotient of that group is the isotropy group of \((p,[g])\in\Lambda X.\)  Indeed, we denote the order of the ineffective kernel of \(Z_{\Gamma_p}(g)\) acting on \(U^g_p\) by \(m(p,[g])\in\f{Z}.\) \(m\) is a locally constant function on \(\Lambda X\) which arises in index theorems for orbifolds.

\subsection{Index Theorems}\label{index}

We can now define the orbifold versions of the \(L\) and \(\hat{A}\) forms, denoted by \(L_\Lambda\) and \(\hat{A}_\Lambda\) since they are differential forms on the inertia orbifold.  We define the forms locally using an orbifold chart.

Let \(\phi_p: U_p\to V_p\) be an orbifold chart centered at \(p\in X\) with isotropy group \(\Gamma_p\). Let \(h\) be a \(\Gamma_p\)-invariant Riemannian metric on \(U_p\).  Then \( U_p^g\) is an orbifold chart for \(\Lambda X\) centered at \((p,[g])\). Since \(U_p^g\) is the fixed point set of \(g\) acting on \(U_p\) we can use \(h\) to define the equivariant \(L\)-form \(L_g(h)\), a differential form on \(U^g_p\) as follows (see \cite{GKS}, \cite {BGV} Chapter 6 and \cite{LM} III.14 for details).  

\(g\) acts on the normal bundle \(NU_p^g\) of \(U_p^g\subset U_p\) by isometries.  At a point \(q\in M^\phi\) there exists a unique decomposition of \(N_qU_p^g\) into subspaces such that \(g\) acts on each subspace as \(e^{i\theta}\), with respect to a complex structure,  for some \(\theta\in (0,\pi)\).   Since \(U_p^g\) is totally geodesic, parallel translation preserves \(NU_p^g\) and commutes with \(g,\) so this decomposition corresponds to a global bundle decomposition 
\[NU_p^g=\bigoplus_jN_j\]
and a complex structure such that \(g|_{N_j}=e^{i\theta_j}\) (for simplicity, we ignore the possibility that \(\theta_j=\pi\) and \(N_{j}\) does not admit a complex structure).  Applying the splitting principle, let 
\[N_j=\bigoplus_k\lambda_{j,k}\] 
be a formal splitting into complex line bundles \(\lambda_{i,j}\) with first Chern forms \(c_1(\lambda_{j,k})=x_{j,k}.\)  Then 
\[L_g(h)=L(TU_p^g)\prod_{j,k}\coth\left(x_{j,k}+\ov{2}i\theta_j\right)\]
is a polynomial in the  Chern forms of the bundles \(N_j\), since it is invariant under permutations in the \(k\) indices.  Note also that giving \(NU_p^g\) the orientation induced by the complex structure, \(U_p^g\) inherits an orientation that is compatible with those on \(NU_p^g\) and \(TU_p.\)

The metric and connections used to define the form will be \(\Gamma_p\) invariant, so \(L_g(h)\) is \(Z_{\Gamma_p}(g)\)-invariant and  defines a differential form on \(\Lambda X\) in this chart.  We define 
\[L_\Lambda(h)=\ov{m(p,[g])}L_g(h).\]   Since \(m\) is globally defined and transition functions between orbifold charts are equivariant, the local definitions of \(L_\Lambda(h)\) are compatible and yield a global differential form on \(\Lambda X\).

Let \(U_p\) be an orbifold chart which trivializes the orthonormal frame bundle. If \(X\) is spin, the derivative action of \(\Gamma_p\) on the trivialization \(U_p\times \SO(n)\) lifts to an action of \(\Gamma_p\) on \(U_p\times \Spin(n).\)  Then

\begin{equation}\label{ahatgdef}\hat{A}_g(h)=\ep(g)\hat{A}(TU_p^g)\prod_{j,k}\ov{2}\text{csch}\left(\ov{2}x_{j,k}+\ov{2}i\theta_j\right)\end{equation}
Where \(\ep(g)\in\{1,-1\}\) depends on the action of \(g\) on \(\Spin(n)\), and the choice of \(\theta_j/2,\) which is defined only up to a an integer multiple of \(\pi.\)  The full definition of \(\ep(g)\) is given in \sref{signs}; for our purposes, we can use the fact that if \(g\in \Gamma_p\) has odd order \(r,\) then \(r\theta_j\in 2\pi,\) and 
\[\ep(g)=\prod_{j,k}\cos\left(\frac{r\theta_j}{2}\right).\]

We can now define
\[\hat{A}_\Lambda(h)=\ov{m(p,[g])}\hat{A}_g(h).\]
Note that on  \(X\subset \Lambda X,\)  \(g=1\), so \(m(p,[1])=1\), \(L_\Lambda(h)=L(h)\), and \(\hat{A}_\Lambda(h)=\hat{A}(h).\)

    We can now state the index theorems for orbifolds with boundary, in the cases we will need.  Let \(W^{4k}\) be an orbifold with boundary \(M.\) Let \(h\) be a product-like metric on \(W\) and let \(g=h|_{M}\).  Define the operator \(B_g\) on \(\Omega^{\text{even}}(M)\) by 
    \[B_g|_{\Omega^{2p}(M)}=(-1)^{k+p+1}(*d-d*)\]
    where \(*\) is the Hodge star.  If \(W\) is spin let \(D_g,D_h\) be the Dirac operators on \(M\) and \(W\) respectively.
    
    Given a self adjoint elliptic operator \(A\) on a closed manifold with spectrum \(\{\lambda_i\}\subset \ar{}\) we define a function of a complex variable \(s\in\f{C}\) by
    \[\eta(A,s)=\sum_{\lambda_i\neq0}\text{sign}(\lambda_i)|\lambda_i|^{-s}.\]
    
    \noindent The function is analytic when the real part of \(s\) is large.  Atiyah, Patodi and Singer showed that \(\eta\) can be analytically continued to a meromorphic function which is analytic at 0.  We define \(\eta(A)=\eta(A,0).\)  The following index theorems are due to Kawasaki and Farsi.  
    
    \begin{thm}\cite{F}, \cite{K}\label{orbiindex}
    	Let \(W\) be an orbifold with boundary.  Let \(h\) a product like metric on \(W\) and \(g=h|_{\partial W}.\) Then  
    \begin{equation}\label{sing}
    \text{\normalfont sign}(W)=\int_{\Lambda W}L_\Lambda(h)-\eta(B_g)
    \end{equation}	
    and if \(W\) is spin
    \begin{equation}\label{ind}
    \text{\normalfont ind}(D_h)=\int_{\Lambda W}\hat{A}_\Lambda(h)-\frac{\text{\normalfont dim}\left(\text{\normalfont  ker}(D_g)\right)+\eta(D_g)}{2}.
    \end{equation}

\end{thm}

\subsection{7 dimensional rational cohomology spheres}\label{exotic}
  Eels and Kuiper \cite{EK} defined a smooth structure invariant \(\mu(M)\in\f{Q}/\f{Z}\) of a closed spin 7-manifold \(M\) as follows.  Assume that the first Pontryagin class \(p_1(TM)=0\in H^4(M,\f{Q}).\)  The 7-dimensional spin cobordism group is trivial, so we can choose a spin manifold \(W\) such that \(\partial W=M.\)  In the long exact sequence 
\[\to H^4(W,M;\f{Q})\to H^4(W,\f{Q})\to H^4(M,\f{Q})\to\]
we have \(p_1(TW)\mapsto p_1(TM)=0\) and so there exists \(p\in H^4(W,M;\f{Q})\) such that \(p\mapsto p_1(TW).\)  Then 
\[\mu(M)=\ov{2^7\cdot 7}\left<p^2,[W,\partial W]\right>-\ov{2^5\cdot 7}\text{sign}(W)\text{ mod }\f{Z}.\]

One checks that \(\mu(M)\) does not depend on the choice of \(W\) nor of \(p\).  \(\mu\) is additive under connected sums, that is \(\mu(M\#N)=\mu(M)+\mu(N)\).  Using the work of Smale, Eels and Kuiper proved that two manifolds homeomorphic to \(S^7\) are orientation preserving diffeomorphic if and only if they have the same value of \(\mu.\)    They showed further that the possible values of \(\mu\) for such a manifold are given by
\[\{\mu(\Sigma^7)|\ \Sigma^7\text{ is homeomorphic to } S^7\}=\ov{28}\{0,1,...,27\}\subset \f{Q}/\f{Z}\]
(where the diffeomorphism types form a group under connected sum).  \(\mu(S^7)=0\) and the Milnor spheres (those which can be written as the total space of an \(S^3\) bundle over \(S^4\)), have \(\mu\) values
\[\ov{28}\{\pm1,\pm3,\pm4,\pm6,\pm7,\pm8,\pm10,\pm11,\pm13,14\}\in\f{Q}/\f{Z}.\]
The \(\mu\) values
\begin{equation}\label{nm}\ov{28}\{\pm2,\pm5,\pm9,\pm12\}\in\f{Q}/\f{Z}\end{equation}
correspond to exotic spheres which are not Milnor spheres.

Crowley completed the diffeomorphism classification of highly connected 7-manifolds.  Here we describe only the case where \(M^7\) is a 2-connected rational homology sphere, so \(\zco{M}{4}\) is finite.

  The \emph{linking form}  lk\(:\zco{M}{4}\times \zco{M}{4}\to\f{Q}/\Z\) is a bilinear form defined as follows.  Let \(\beta:H^3(M,\Q/\Z)\to H^4(M,\Z)\) be the Bockstein homomorphism corresponding to the coefficient sequence \(\f{Z}\to\Q\to\Q/\Z.\)  Let \(x,y\) be torsion classes in \(\zco{M}{4}\).  Since \(x\) has finite order, there exists \(\hat{x}\in H^3(M,\Q/\Z)\) such that \(\beta(\hat x)=x.\)  Then lk\((x,y)=\left<\hat x\cup y,[M]\right>\in \Q/\Z.\)  The isomorphism type of the linking form is a homotopy invariant.  
  
  We note that the total space \(M\) of an \(S^3\) bundle over \(S^4\) with nonzero Euler number \(n\) is a 2-connected rational cohomology sphere with \(\zco{M}{4}\cong\f{Z}_{|n|}.\)  The linking form of such a space is always \emph{standard}, meaning it is isomorphic to the bilinear form 
  \[\f{Z}_{|n|}\times \f{Z}_{|n|}\to\Q/\Z\]
  \[(x,y)\mapsto \frac{xy}{|n|}.\]
  The linking form can be used to determine if a 2-connected rational cohomology sphere, such as many of the manifolds \(M_\du{a}{b}\), has the homotopy type of an \(S^3\) bundle over \(S^4\). 
  
  Let \(\frac{p_1}{2}(TM)\) be the spin characteristic class (the pullback of the generator of \(\zco{B\Spin}{4}\)) which has the property that \(2\frac{p_1}{2}(TM)=p_1(TM)\).  In \cite{Cr} Crowley defines a function \(q_M:\zco{M}{4}\to\Q/\Z\) such that for all \(x,y\in\zco{M}{4}\)
\begin{equation}\label{qprops}q_M(x+y)=q_M(x)+q_M(y)+\mathrm{lk}(x,y)\end{equation}
\[q_M(x)-q_M(-x)=lk\left(x,\frac{p_1}{2}(TM)\right).\]

Crowley then proved that given two 2-connected rational cohomology spheres \(M\) and \(N\) there exists an isomorphism \(\theta:\zco{M}{4}\to\zco{N}{4}\) such that \(q_N\circ\theta=q_M\) if and only if  there exists a manifold \(\Sigma^7\) homoemorphic to \(S^7\) such that \(M\) is diffeomorphic to the connected sum \(N\#\Sigma\) .

If we know further that \(\mu(M)=\mu(N),\) we can conclude that 
\[\mu(M)=\mu(N\#\Sigma)=\mu(N)+\mu(\Sigma)\]
and thus \(\mu(\Sigma)=0,\) \(\Sigma\) is diffeomorphic to \(S^7,\) and \(M\) is diffeomorphic to \(N.\)  Thus \(q\) and \(\mu\) form a complete set of diffeomorphism invariants for 2-connected rational cohomology spheres.

Note that if \(\zco{M}{4}\) has odd order, \(lk\) and \(\frac{p_1}{2}\) uniquely determine \(q_M\)  by \eqref{qprops}.  Furthermore \(p_1(TM)\) uniquely determines \(\frac{p_1}{2}(TM).\)  Thus in that case the existence of a diffeomorphism can be determined using \(\mu\) and an  isomorphism preserving lk and \(p_1.\)    

\subsection{Kreck-Stolz Invariant}\label{ksinv} In \cite{KS}, Kreck and Stolz defined a refinement of \(\mu\) which depends on a choice of metric on \(M\).   Let \(g\) be a Riemannian metric of  positive scalar curvature on a closed spin 7-manifold \(M\) with vanishing rational first Pontryagin class.  Let \(D_g\) be the Dirac operator and \(B_g\) the operator on \(\Omega^\text{even}(M)\) defined in \sref{index}.  Let \(p_1(g)\) be the Chern-Weil Pontryagin form defined in terms of the curvature tensor of \(g\) and let \(\overline{p_1(g)}\) be any 3-form such that \(d\overline{p_1(g)}=p_1(g).\)  Such a 3-form exists by the assumption on the first Pontryagin class of \(M.\)  Then

\begin{equation}\label{sdef}
s(M,g)=-\ov{2}\eta(D_g)-\ov{2^5\cdot 7}\eta(B_g)+\ov{2^7\cdot 7}\int_Mp_1(g)\overline{p_1(g)}
\end{equation}
where  \(\eta\) is the spectral invariant described in \sref{index}.  Using the Atiyah-Patodi-Singer index theorem applied to a metric on \(W\) extending \(g\), Kreck and Stolz showed that
\begin{equation}\label{smu}\mu(M)=s(M,g)\text{ mod }\f{Z}.\end{equation}
They showed further that for a simply connected manifold \(M,\) \(|s(M,g)|\) depends only on the connected component of the image of \(g\) in \(\mods{scal}(M).\)  Later, using the fact due to Bohm and Wilking \cite{BW} that a nonnegatively curved metric on a simply connected manifold evolves immediately to a metric of positive Ricci curvature under the Ricci flow,   Belegradek, Kwasik and Schultz \cite{BKS}, Dessai, Klaus and Tuschmann \cite{DKT}, and Belegradek and Gonz\'{a}lez-\'{A}lvaro \cite{BeGo} proved the following: 

\begin{lem}\label{comps}
	Let \(g_1,g_2\) be Riemannian metrics with sec\((g_i)\geq0\) and scal\((g_i)>0\) on a spin manifold \(M^{4k-1}\) with vanishing rational Pontryagin classes.  If \(|s(M,g_1)|\neq|s(M,g_2)|\), then \(g_1,g_2\) represent distinct connected components of \(\modse{sec}(M)\).  
\end{lem} 
\noindent By \cite[Prop. 2.13.i]{KS}, the quantity \(|s|\)  is  invariant under pullbacks by diffeomorphisms of simply connected manifolds.  That is, for simply connected spin manifolds \(M,N\), if \(g\) is a metric of positive scalar curvature on \(M\) and \(\phi:M\to N\) is a diffeomorphism, then \(|s(M,\phi^*g)|=|s(N,g)|\).    To prove \(\modse{sec}(M)\) has infinitely many components it suffices to find an infinite set of manifolds diffeomorphic to \(M\), each with a metric of nonnegative curvature, such that the corresponding set of values of \(s\) is infinite.  

\subsection{Cohomogeneity one actions}\label{coho1prelim}

Let \(G\) be a compact Lie group with closed subgroups \(H,K_+,K_-\) such that \(H\subset K_{\pm}\) and \(K_\pm/H\) is diffeomorphic to a sphere \(S^{d_\pm-1}.\)  Then the action of \(K_\pm\) on \(K_\pm/H\) extends linearly to an action on the disc \(D^{d_\pm}\) and we construct two disc bundles
\[P_\pm=G\times_{K_\pm}D^{d_{\pm}}.\]   
\(D_\pm\) has a natural left action by \(G\) and \(G\backslash D_\pm\) is diffeomorphic to an interval.  Note that 
\[\partial D_\pm\cong G\times_{K_\pm}(K_\pm/H)\cong G/H\]
and all the diffeomorphisms are equivariant.  Thus we can define a closed manifold 
\[P=P_-\cup_{G/H}P_+\]
with a left action by \(G.\) The quotient \(G\backslash P\) is diffeomorphic to \([-1,1]\) and the action of \(G\) on \(P\) is said to be of {\emph {cohomogeneity one}}.  Let \(\pi:P\to[-1,1]\) be the quotient map; then \(\pi^{-1}(\pm1)\cong G/K_\pm\) is the zero section of \(D_\pm\).    Those orbits have isotropy type \(K_\pm\) and the orbits in \(\pi^{-1}((-1,1))\) have isotropy type \(H.\)  Indeed, any closed manifold admitting an action by a compact Lie group with quotient \([-1,1]\) is of this form.   Grove and Ziller proved the following:

\begin{thm}\cite{GZ}\label{GZ}
	If \(d_+=d_-=2,\) then \(P\) admits a \(G\)-invariant metric with nonnegative sectional curvature.   
\end{thm}

\begin{section}{Kreck-Stolz Invariant of the boundary of an Orbifold}\label{ks}
In this section we derive formulae to compute the Kreck-Stolz \(s\) invariant  for a manifold \(M^7\) which is given as the boundary of a spin orbifold, and then specifically as the boundary of the total space of an orbifold disc bundle.   We present the 7-dimensional case for simplicity, although a similar derivation works for any dimension of the form \(4k-1.\)

   When \(M\) is the boundary of a spin manifold \(W\) and \(h\) is a product-like metric on \(W\) with \(h|_{M}=g\), Kreck and Stolz used the Atiyah-Patodi-Singer index theorem for manifolds with boundary to express \(s(M,g)\) in terms of the index of the Dirac operator  and topological invariants of \(W\).  The following lemma generalizes Proposition 2.13 (iii) in \cite{KS} to the case where \(W\) is an orbifold.

\begin{lem}\label{sbound}

	Let \(M^7\) be a closed manifold with vanishing first real Pontryagin class.  Let \(W^8\) be a spin orbifold with \(\partial W=M.\) Let \(h\) be a product like metric on \(W\) and let \(g=h|_{M}.\)  If scal\((g)>0\) then

	\begin{equation}\label{sboundeq}
	s(M,g)=\text{\normalfont ind}(D_h)+\ov{2^5\cdot 7}\text{\normalfont sign}(W)-\ov{2^7\cdot 7}\int_Wp\wedge q-\int_{\Lambda_sW}\left(\hat{A}_\Lambda(h)+\ov{2^5\cdot 7}L_\Lambda(h)\right)	\end{equation}	for any \(p,q\in p_1(TW)\) such that \(q|_M=0.\)

 	\end{lem}
 
 \begin{proof}
 	Taking a linear combination of \eqref{sing} and \eqref{ind} and evaluating the \(\hat{A}\) and \(L\) classes in degree 8 we have 
 	\[\text{ind}(D_h)+\ov{2^5\cdot 7}\text{sign}(W)=\int_{\Lambda W}\left(\hat{A}_\Lambda+\ov{2^5\cdot 7}L_\Lambda\right)-\ov{2}\eta(D_g)-\ov{2^5\cdot 7}\eta(B_g)\]
 	\begin{equation}\label{indexp}=\ov{2^7\cdot 7}\int_Wp_1(h)\wedge p_1(h)+\int_{\Lambda_sW}\left(\hat{A}_\Lambda+\ov{2^5\cdot 7}L_\Lambda\right)-\ov{2}\eta(D_g)-\ov{2^5\cdot 7}\eta(B_g).\end{equation}
 	Here we use that scal\((g)>0\) so \(\text{ker}(D_g)=\{0\}\).

Since \(h\) is product-like near \(M\) and \(p_1(g)\) is exact by the assumption on the first Pontryagin class of \(M\), \(p_1(h)|_{M}=p_1(g)=d\overline{p_1(g)}\) for some form \(\overline{p_1(g)}\) on \(M.\) The following equation, which is a version of \cite{KS} Lemma 2.7, follows from Stokes' theorem:

	\[\int_Wp_1(h)\wedge p_1(h)=\int_Wp\wedge q+\int_{\partial W}p_1(g)\wedge \overline{ p_1(g)}.\]

 \noindent Using the definition \eqref{sdef} we have
\[\text{ind}(D_h)+\ov{2^5\cdot 7}\text{sign}(W)=\ov{2^7\cdot 7}\int_Wp\wedge q+\int_{\Lambda_sW}\left(\hat{A}_\Lambda+\ov{2^5\cdot 7}L_\Lambda\right)+s(M,g)\]
 completing the proof of the \lref{sbound}.  
 \end{proof}
	
	Recall from \sref{dirac} that if scal\((h)>0\), as will be the case for the examples in \sref{mfds}, we can conclude that ind\((D_{h})=0\).  The manifolds used to prove \tref{thm} are total spaces of orbifold 3-sphere bundles and thus the boundaries of total spaces of orbifold 4-disc bundles.  We reformulate two of the terms in \eqref{sboundeq} in that case to facilitate computation.
	
	\begin{lem}\label{sbase}

 Let \(B^4\) be a closed orbifold and \(E^8\to B^4\) an oriented rank 4 orbifold vector bundle with non-vanishing Euler number.  Let \(W^8\subset E^8\) be the corresponding 4-disc bundle.  Then 
   
	\[\text{\normalfont sign}(W)=\text{\normalfont sign}\left(\int_Be(E)\right)\]
	where the right side refers to the sign of a nonzero real number.  Furthermore, there exist \(p,q\in p_1(TW)\) such that \(q|_{\partial W}=0\) and 
	\begin{equation}\label{sbaseeq}
\int_Wp\wedge q=\left(\int_Bp_1(TB)+\int_Bp_1(E)\right)^2\left(\int_Be(E)\right)^{-1}.\end{equation}

	\

\begin{proof}

  Let \(\pi:W\to B\) be the orbifold bundle projection.   Then 
	\[TW=\pi^*(TB\oplus E).\]
	Choosing any connections \(\nabla_{TB}\) on \(TB\) and \(\nabla_E\) on \(E\) we  set
	\[p=\pi^*p_1(\nabla_{TB})+\pi^*p_1(\nabla_E).\]
	Let \(\Phi\in\Omega^4(W)\) be a Thom form for \(\pi:W\to B,\) so \(\Phi|_{\partial W}=0\) and \([\Phi]=\pi^*e(E).\)  Define   
	\[p_B=\frac{\int_Bp_1(TB)}{\int_Be(E)}\]
	\[p_E=\frac{\int_Bp_1(E)}{\int_Be(E)}\]
	so \(p_Be(E)= p_1(TB)\) and \(p_Ee(E)= p_1(E).\) Then choose	\(q=(p_B+p_E)\Phi.\)   Integrating over the fibers (see \eqref{iof}) we have 
	\[\int_Wp\wedge q=(p_B+p_E)\int_W\pi^*(p_1(\nabla_{TB})+p_1(\nabla_{E}))\wedge\Phi=(p_B+p_E)\int_B(p_1(TB)+p_1(E))\]
\[	=\left(\int_Bp_1(TB)+\int_Bp_1(E)\right)^2\left(\int_Be(E)\right)^{-1}
	.\]

	It remains to compute sign\((W).\)  By the Thom isomorphism  \(H^4(W,\partial W; \ar{})\) is isomorphic to \(\ar{}\), generated by \([\Phi].\)  
Thus the intersection form 
	is determined by
	\[([\Phi],[\Phi])\mapsto \int_W\Phi\wedge\Phi=\int_W\pi^*e(E)\wedge\Phi=\int_Be(E).\]
The signature of this bilinear form is the sign of \(\int_Be(E).\)

\end{proof}

		\end{lem}

	\end{section}

\section{Inertia Orbifolds and Index Theory of quotients}\label{quotientinertia}
This section and the following are dedicated to developing tools for computing \(\int_{\Lambda_sX}L_\Lambda\) and \(\int_{\Lambda_sX}\hat{A}_\Lambda\) for an orbifold \(X\). In this section, using a description of \(X\) as the quotient of a manifold \(M\) by the almost free action of  a Lie group \(G\) we give a global description of each component of the integrals.  A related description of the inertia orbifold is given in \cite[Theorem 3.14]{ALR}.

Let a compact connected Lie group \(G\) act almost freely on a compact oriented manifold \(M.\)  Give \(G\) a biinvariant metric, \(M\) a \(G\)-invariant metric, and \(G\times M\) the product metric.  Let \(G\) act isometrically on the left of \(G\times M\) by \(g(h,p)=(ghg^{-1},gp)\) for \(g,h\in G\) and \(p\in M\).  Define a \(G\)-equivariant diffeomorphism \(T:G\times M\to G\times M,\) \(T(g,p)=(g,gp)\) for \(g\in G\) and \(p\in M\).

\begin{lem}\label{lgm1}
	Let \(\bar I=(G\times M)^T\) be the fixed point set of \(T.\)  Then \(\bar I\) is an embedded \(G-\)equivariant submanifold of \(G\times M\) and 
\(G\bs \bar{I}\) is diffeomorphic to  \(\Lambda(G\bs M)\).  
\end{lem}
\begin{proof} Let \((g,p)\in\bar I.\)  Let \(V_p,H_p\subset T_pM\) be the vertical and horizontal subspaces of the \(G\) action on \(M\).  Using the exponential map on \(M,\) we can choose a slice neighborhood \(U_p\subset M\) centered at \(p\) such that \(T_pU_p=H_p.\)

	Then \(G\times (G\times_{G_p}U_p)\) can be equivariantly identified with  an open set  \(Y\subset G\times M\). 
	\(\bar I \cap Y\) is then identified with \[\{(hg h^{-1},[h,x])|h\in G,g\in G_p,x\in U_p^g\}. \]
	
	For each \(g\in G_p,\) let \((g)\) denote the \(G_p-\) conjugacy class of \(g.\)  Choosing a representative of each such conjugacy class we define 
	\begin{equation}\label{embedding}\coprod_{(g)\subset G_p}G\times_{Z_{G_p}(g)}(\{g\}\times U_p^\gamma)\to  Y\end{equation}
	\[[h,(g,x)]\mapsto (hg h^{-1},[h,x])\mapsto (hgh^{-1},hx).\]
	
	One checks that the map is a \(G-\)equivariant injective immersion with image \(\bar I\cap Y,\) and thus \(\bar I\) is an embedded submanifold.

	Let \(X=G\bs M\).  We define a diffeomorphism \(\phi:\Lambda X\to I\) locally on orbifold charts for \(\Lambda X\).  
	Recall that \(U_p^g\) is an orbifold chart for \(\Lambda X\) with isotropy group \(Z_{G_p}(g)\).  The embedding \eqref{embedding} implies that that \(\{g\}\times U_p^g\subset \bar I\) is a slice neighborhood with respect to the \(G\) action on \(\bar I,\) and thus an orbifold chart for \(I.\) 
	
	Thus \(\tilde\phi:U_p^g\to\{g\}\times U_p^g, \tilde\phi(x)=(g,x),\) is a \(Z_{G_p}(g)-\)equivariant diffeomorphism between the two orbifold charts, and \(\pi\circ \tilde\phi:U_p^g\to I\) locally defines a diffeomorphism \(\phi: \Lambda X\to I.\)  
	
	We must check that the local definitions are compatible and give an injective map \(\phi\).  This is equivalent to the statement that \(x\in U_p^g\) and \(y\in U_q^h\) represent the same point of \(\Lambda X\)  if and only if \([g,x]=[h,y]\in I.\)  
	
	\(x,y\) represent the same point in \(\Lambda X\) only if they are related by one of the open embeddings of orbifold charts used to define the orbifold structure; see \sref{orbifolds}.  In one direction, we may assume therefore that there is an open embedding \(\psi:U_q\to U_p\) such that \(\psi(y)=x,\) where \(\psi(z)=f(z)z\) for a continuous function \(f:U_q\to G\) such that \(g=f(q)hf(q)^{-1}\).  As  discussed in \sref{orbifolds}, \(f(hz)hf(z)^{-1}=f(q)hf(q)^{-1}=g\) for all \(z\in U_q.\)  Since \(h\) fixes \(y,\) 
	\[g=f(hy)hf(y)^{-1}=f(y)hf(y)^{-1}.\]
	If follows that 
	\[[g,x]=[f(y)hf(y)^{-1},f(y)y]=[h,y].\]
	
	Conversely, assume that \([g,x]=[h,y],\) so \(g=khk^{-1}\) and \(x=ky\) for some \(k\in G.\)  Since \(ky\in U_p,\) we can choose a slice neighborhood \(U_y\subset U_q\) and a function \(f:U_y\to G\) such that \(f(y)=k\) and \(f(z)z\in U_p\) for all \(z\in U_y, \) which induces an open embedding \(\phi:U_y^h\to U_p^{khk^{-1}}=U_p^g\)  such that \(\phi(y)=f(y)y=x.\)  Together with the inclusion \(U_y^h\hookrightarrow U_q^h\) this embeddings identifies \(x\) and \(y\) in \(\Lambda X.\)  This proves the existence of the diffeomorphism \(\phi.\)  \end{proof}

Let \(\bar I=\cup_i\bar I_i\) be the decomposition into connected components, which may have different dimensions. Let \(\operatorname{Proj}_{TM}\) be the orthogonal projection onto \(\{0\}\oplus TM\subset TG\times TM.\)  Since the action of \(G\) on \(M\) is almost free, \(\operatorname{Proj}_{TM}\) is injective on \(T\bar I_i.\)  Let \(\bar N_i\) be the orthogonal complement in \(\{0\}\oplus TM\) of  \(\operatorname{Proj}_{TM}(T\bar I_i)\).   Then \(\bar N_i\to \bar I_i\) is a \(G-\) equivariant vector bundle , and the derivative \(dT:TG\times TM\to TG\times TM\) preserves \(\bar{N}_i\).   There is a \(G-\)invariant decomposition 
\[\overline N_i=\overline N_{i,\pi}\oplus\bigoplus_{j}\overline N_{i,j}\]  
such that \(dT|_{\overline N_{i,\pi}}=-1\) and \(dT|_{\overline N_{i,j}}=e^{\sqrt{-1}\theta_j}\) with respect to a \(G\)-invariant complex structure on \(\bar N_{i,j}.\) See the proof of \lref{lgm} for details.    For our purposes, we assume \(\bar N_{i,\pi}=\{0\}\) for all \(i.\)  

Let \(c( N_{i,j})=\prod_{k}(1+x_{i,j,k})\) be a formal splitting of the Chern class of the complex orbifold vector bundle \(N_{i,j}=G\bs \bar N_{i,j}\to I_i=G\bs \bar I_i.\)  Choose \((g,p)\in\bar I_i\), let \(G_p\) be the stabilizer at \(p\) of the \(G\) action on \(M\), and define \(C_i\) to be the order of the ineffective kernel of the centralizer \(Z_{G_p}(g)\) on the horizontal space at \((g,p)\) of the \(G\) action on \(\bar I_i.\)

We orient \(\bar I_i\) such that \(\operatorname{Proj}_{\{0\}\oplus TM}(T\bar I_i)\oplus \bar N_i\) induces the orientation on \(TM\) and orient \(I_i\) as the quotient of \(\bar I_i.\)    
  
\begin{lem}\label{lgm}
	\[\int_{\Lambda(G\bs M)}L_\Lambda=\sum_{i}\ov{C_i}\int_{I_i}L(TI_i)\prod_{j,k}\coth\left(x_{i,j,k}+\sqrt{-1}\frac{\theta_{i,j}}{2}\right).\]
	
	and if \(G\bs M\) is spin, 
	\[\int_{\Lambda(G\bs M)}\hat{A}_\Lambda=\sum_{i}\frac{\ep_i}{C_i}\int_{I_i}\hat{A}(TI_i)\prod_{j,k}\text{\normalfont csch}\left(\frac{x_{i,j,k}}{2}+\sqrt{-1}\frac{\theta_{i,j}}{2}\right)\]
	for \(\ep_i=\pm1.\)
	\end{lem}

	

We give a full definition of  \(\ep_i\) in \sref{signs} ; for our purposes, we will be able to use the following proscription:  if \((g,p)\in\bar I_i\) and \(g\) has odd order \(r,\) then \(r\theta_{i,j}\in2\pi\f{Z}\) and
\(\ep_i=\prod_{j,k}\cos\left(\frac{r\theta_{i,j}}{2}\right).\)  One can consider the case \(\bar N_{i,\pi}\neq 0\) similarly, but the added complexity is unnecessary for our purposes.

\begin{proof} 
	
	Observe that 
	\begin{equation}\label{ti}T_{(g,p)}\bar I=(T_{(g,p)}G\times M)^{dT}=\{(v,w)\in T_gG\times T_pM |v^*(p)+dg(w)=w\}\end{equation}
	where \(v^*\) is the action field on \(M\) corresponding to the right invariant vector field on \(G\) extending \(v\).  Let \(H_p^{g}\) the fixed point set of \(g\) acting on \(H_p\), and \((H_p^{g})^\perp\) the orthogonal complement of \(H_p^{g}\) in \(H_p\).  We will use several times the identification, which follows from \eqref{ti}, 
	\begin{equation}\label{projcomp}\text{Proj}_{T_pM}(T_{(g,p)}\bar I)=V_p\oplus H_p^{g}\text{ \ \ and \ \ } \overline N_{(g,p)}=(H_p^{g})^\perp.\end{equation}


	Just as in the proof of \lref{lgm1}, let \(U_p\) be a slice neighborhood of \(M\) such that \(T_pU_p=H_p\).  Let \(NU_p^g\) be the normal bundle of \(U_p^g\subset U_p\).  Define a vector bundle
	\[G\times_{Z_{G_p}(g)}NU_p^g\to G\times_{Z_{G_p}(g)}U_p^g\]
	and a bundle map \(\Phi:G\times_{Z_{G_p}(g)}NU_p^g\to \overline N_i|_{\bar I\cap Y}\) covering the embedding \eqref{embedding} by identifying \(v\in N_xU_p^g\subset T_xM\subset T_{(g,x)}G\times M\) and setting 
	\[\Phi([h,v])=\text{Proj}_{\overline N_{i}}(hv).\]
	Since \(T_pU_p^g=H_p^{g},\) by \eqref{projcomp} \(N_pU_p^g=(H_p^{g})^\perp.\)  Observing that \(h(H_p^g)^\perp=(H_{hp}^{hgh^{-1}})^\perp=\bar N_{(hgh^{-1},hp)}\) we see that \(\Phi\) is an isomorphism on the fiber over \([h,p].\) Shrinking \(U_p\) if necessary we can assume \(\Phi\) is an isomorphism.  Since \(hgh^{-1}\) acts by isometries on \(T_{hx}M\) , preserves \(\overline{N}_{i}|_{(hgh^{-1},hx)},\) and is equivalent to the action of \(dT\) on \(\overline{N}_{i}|_{(hgh^{-1},hx)},\)
	\begin{equation}\label{equivariance}\Phi([h,gv])=\operatorname{Proj}_{\bar N_{i}}(hg v)=hgh^{-1}\operatorname{Proj}_{\bar N_{i}}(hv)=dT\Phi([h,v])\end{equation}

	Recall from \sref{index} the decomposition \(NU_p^g=\oplus_{j}N_{g,j}\) such that \(g|_{N_{g,j}}=e^{\sqrt{-1}\theta_{g,j}}\), \(\theta_{g,j}\in(0,\pi)\) \(\theta_{g,1}< \theta_{g,2}<...\), with respect to a complex structure on \(N_{g,j}.\)  The decomposition and complex structures are \(Z_{G_p}(g)-\)invariant, and extends to decomposition
	\[G\times_{Z_{G_p}(g)}NU_p^g=\oplus_jG\times_{Z_{G_p}(g)}N_{g,j}.\]  On \(\overline I_i\cap Y\), define \(\bar N_{g,j}=\Phi(G\times_{Z_{G_p}(g)}N_{g,j}).\)  \(\Phi\) induces a complex structure on \(\bar N_{g,j}\), and by \eqref{equivariance} \(dT_{\bar N_{g,j}}=e^{i\theta_{g,j}}\).    Since the decomposition of \(\bar N_i\) and the complex structures are uniquely determined by \(dT,\) they do not depend on the choice of local isomorphism \(\Phi.\)  Furthermore, the quantities \(\theta_{g,j}\) are locally constant and thus constant on each component \(\bar I_i;\) we relabel them \(\theta_{i,j}\), and relabel \(\bar N_{i,j}\), accordingly (that is, if \((g,p),(g',p')\in \bar I_i\) it must be the case that \(\theta_{g,j}=\theta_{g',j}=:\theta_{i,j}\).)  Since \(dT\) is \(G-\)equivariant, the decompositions and complex structures are as well.

	The local diffeomorphism \(\tilde \phi:U_p^g\to\{g\}\times U_p^g\) between orbifold charts of \(\Lambda X\) and \(I_i\) used to define \(\phi\) in the proof of \lref{lgm1} is covered by \(\Phi|_{U_p^g}:N_{g,j}\to\overline{N}_{i,j}|_{\{g\}\times U_p^g}\) where \((g,p)\in\bar I_i\). By definition \(\overline{N}_{i,j}|_{\{g\}\times U_p^g}\) is the pullback to this chart of \(N_{i,j}=G\bs \bar N_{i,j}.\)  Thus considered as local forms on this chart, 
	\[c(\bar N_{i,j})=c(N_{i,j})=\prod_k(1+x_{i,j,k})\]
	and \(c(N_{g,j})=\prod_k(1+\tilde\phi^*(x_{i,j,k}))\).  Similarly, identified as forms in this chart  \(L(TI_i)=L(T\{g\}\times U_p^g)\) which  pulls back under \(\phi\) to \(L(TU_p^g)\).  Recalling that our relabeling is such that \(\theta_{g,j}=\theta_{i,j}\), by the definition of \(L_\Lambda\)
	
		\begin{align*}L_\Lambda|_{U_p^g}=\ov{m(([p],g))}L_g&=\tilde\phi^*\left(\ov{m(([p],g))}L(TI_i)\prod_{j,k}\left(x_{i,j,k}+\sqrt{-1}\frac{\theta_{i,j}}{2}\right)\right)\end{align*}

	Furthermore, since \(T_pU_p^g=H_p^{g}\) is exactly the horizontal part of \(T\bar I_{(g,p)}\) (see \eqref{ti}), \(m(([p],g)),\) which is defined as the order of the ineffective kernel of \(Z_{G_p}(g)\) acting on \(U_p^g\), is  equal to \(C_i\), the order of the ineffective kernel of the action on \(H_p^{g}.\)  Since \(\tilde\phi\) is the local definition of \(\phi,\)
	
	\[L_\Lambda|_{I_i}=\phi^*\left(\ov{C_i}L(TI_i)\prod_{j,k}\left(x_{i,j,k}+\sqrt{-1}\frac{\theta_{i,j}}{2}\right)\right).\]
	
	The first formula in the lemma follows once we confirm that \(\phi\) is orientation preserving.  Since \(\tilde\phi\) is essentially the identity, this amounts to showing that the orientations defined on \(\Lambda X\) and \(I\) induce the same orientation on \(H_p^g=T_pU_p^g\).     The orientation on \(\Lambda X\) (see \sref{index}), is defined such that the orientation on \(H_p^g\), the orientation induced by the complex structure on \(N_pU_p^g={H_p^g}^\perp\), and the natural orientation on \(V_p\cong\mathfrak{g}\) induce the orientation on 
	\[T_pM=(H_p^g\oplus (H_p^g)^\perp)\oplus V_p.\]
	If \(V_{(g,p)}^{\bar I} \subset T_{(g,p)}\bar I\) is the vertical bundle of the \(G\) action on \(\bar I,\) then Proj\(_{T_pM}(V_{(g,p)}^{\bar I})=V_p,\) and Proj\(_{T_pM}\) commutes with the identification with \(\mathfrak{g}.\)  Then by the definition of the orientation on \(I,\) the orientation on \(H_p^g\) is induced by the isomorphism 
	\[T_pM=\text{Proj}_{T_pM}(T_{(g,p)}\bar I)\oplus\overline{N}_{(g,p)}\cong( H_p^g\oplus V_p)\oplus (H_p^g)^\perp,\] see \eqref{projcomp}.  Since \((H_p^g)^\perp\) admits a complex structure and is even dimensional, the two orientations on \(H_p^g\) agree, and \(\phi\) is is orientation preserving.

	
	

	The final equation follows similarly, after noting that \(\ep(g)=\ep_i\); see \sref{signs} for the details. 

\end{proof}

\section{Inertia orbifolds from Cohomogeneity 1 actions}\label{coho1inertia}

In this section we apply \lref{lgm1} and  \lref{lgm} to describe the inertia orbifolds of quotients \(W\) of certain almost free subactions of cohomogeneity one actions, and orbifold bundles associated to those  actions.  The integrals  \(\int_{\Lambda_sW}L_\Lambda\) and \(\int_{\Lambda_sW}\hat{A}_\Lambda\) can then be computed in therms of the data of the cohomogeneity one actions.  

Let a compact connected Lie group \(G\) act with cohomogeneity one on a closed manifold P with group diagram
\(H\subset K_\pm \subset G,\) as described in \sref{coho1prelim}.  
Let \(J\) be a connected normal subgroup of \(G\) and \(V^n\) a vector space equipped with a representation of \(G.\)   We consider the inertia orbifold of \(W=P\times_{J}D^n\) (with boundary, if \(n>0\))  where \(D^n\subset V^n\) is the unit disc, under the assumptions:  

\begin{enumerate}
	\item \(H\) is finite, \(J\cap H=\{1\}\) and \(J\cap \K_\pm=\Gamma_\pm\) is finite group of odd order contained in \(K_\pm^0\), the connected component of the identity in \(K_\pm.\)
	\item  \(\Gamma_\pm\) acts freely on \(V\bs \{0\}\). 
	\item The centralizer \(Z_{K_\pm}(\gamma)=K_\pm^0\) for all nontrivial \(\gamma\in\Gamma_\pm\)
\end{enumerate}

Note that since \(J\)  is normal and \(\Gamma_\pm\) is finite it follows that any element of \(\Gamma_\pm\) is in the center of \(K_\pm^0,\) and thus \(\Gamma_\pm\) is abelian.  Indeed (3) is an assumption about the elements of \(K_\pm\bs K_\pm^0.\)

Recall from \sref{coho1prelim} that \(K_\pm\) acts on \(\ar{d_\pm}\) by extension of the action on \(K_\pm/H\cong S^{d_\pm-1}\).   Since \(J\cap H\) is trivial and \(J\) is normal \(\Gamma_\pm\) acts freely on \(\ar{d_\pm}\bs\{0\}.\)  It follows from representation theory that \(\ar{d_\pm}\) and  \(V\) admit complex structures such that the  representations of \(K_\pm^0\) are complex, and decompose into complex \(K_\pm^0\)-invariant subspaces \(\ar{d_\pm}\oplus V=\oplus_jV_{\pm,j}\) such that for \(\gamma\in\Gamma_\pm\), \(\gamma|_{V_{\pm,j}}=e^{i\theta_{\pm,j}(\gamma)}.\)   Let \(S_\pm\subset \Gamma_\pm\bs\{1\}\) contain a unique element of every nontrivial conjugacy class of \(K_\pm\) intersecting \(\Gamma_\pm.\)  

Let \(F\) denote the Lie group \(J\bs G\).  Then \(K^0_\pm\) acts on the right of \(F\) with isotropy \(\Gamma_\pm\) and quotient the smooth homogeneous space \(B_\pm=F/(K_\pm^0/\Gamma_\pm)\).     Finally, let \(c(V_{\pm,j})=\prod_{k}(1+x_{\pm,j,k})\) be the formal decomposition of the Chern class of the orbifold complex vector bundle (with smooth base)
\[F\times_{K^0_\pm}V_{\pm,j}\to B_\pm.\]

Let \(\mathfrak{j},\mathfrak{k}_\pm,\mathfrak{g},\mathfrak{f}\) denote the Lie algebras of \(J,K_\pm,G\) and \(F\) respectively.  Since \(J\cap K_\pm\) is finite, \(\mathfrak{k_\pm}\subset\mathfrak{f},\) and let \(\mathfrak{k}_\pm^\perp\subset \mathfrak{f}\) be the orthogonal complement.   We orient \(T_{[1]}F/K_\pm^0\cong \mathfrak{k}_\pm^\perp\) such that \(\mathfrak{k}_\pm^\perp\oplus\mathfrak{j}\oplus\mathfrak{k}_\pm\) induces the orientation on \(\mathfrak{g}\).



Away from the singular orbits \(\pi^{-1}(\pm1)\) (see \sref{coho1prelim} for notation ), the orbits of \(P\) all have isotropy type \(H\), and form an open dense set equivariantly diffeomorphic to \(G/H\times (-1,1)\).  Thus \(P\) inherits an orientation from \(\mathfrak{g}\oplus\text{span}(\partial _r)\) where \(r\) is the coordinate on \((-1,1).\)  Giving \(V\) the orientation compatible with the complex structure, \(W\) inherits an orientation as a quotient of \(P\times D^n.\)     

We will discuss two orientations on \(\ar{d_\pm}\).  One is induced by the complex structure discussed above.  The other is induced by the identification by means of the \(K_\pm\) action, for \(x\neq0\), 
\[T_x\ar{d_\pm}=\text{span}(\partial_r)\oplus\mathfrak{k}_\pm\] 
where \(r\) is the radial coordinate on \(\ar{d_\pm}\).  

	
	
\begin{thm}\label{coho1}
	The set of connected components of \(\Lambda_sW\) is in bijection with \(S_-\cup S_+\), and each component is diffeomorphic to \(B_-\) or \(B_+\).  Furthermore, 
	\begin{align*}\int_{\Lambda_sW}L_\Lambda=&-\frac{\sigma_-}{|\Gamma_-|}\sum_{\gamma\in S_-}\int_{B_-}L(TB_-)\prod_{j,k}\coth\left(x_{-,j,k}+\frac{i\theta_{-,j}(\gamma)}{2}\right)\\&+\frac{\sigma_+}{|\Gamma_+|}\sum_{\gamma\in S_+}\int_{B_+}L(TB_+)\prod_{j,k}\coth\left(x_{+,j,k}+\frac{i\theta_{+,j}(\gamma)}{2}\right)\end{align*}
	Here \(\sigma_\pm=1\) if the two orientations on \(\ar{d_\pm}\) are the same.  
	
	If \(W\) is spin, 
	\begin{align*}\int_{\Lambda_sW}\hat{A}_\Lambda=&-\frac{\sigma_-}{|\Gamma_-|}\sum_{\gamma\in S_-}\ep_-(\gamma)\int_{B_-}\hat{A}(TB_-)\prod_{j,k}\text{\normalfont csch}\left(x_{-,j,k}+\frac{i\theta_{-,j}(\gamma)}{2}\right)\\&+\frac{\sigma_+}{|\Gamma_+|}\sum_{\gamma\in S_+}\ep_+(\gamma)\int_{B_+}\hat{A}(TB_+)\prod_{j,k}\text{\normalfont csch}\left(x_{+,j,k}+\frac{i\theta_{+,j}(\gamma)}{2}\right)\end{align*}
	where 
	\(\ep_\pm(\gamma)=\prod_{j,k}\cos\left(\frac{r\theta_{\pm,j}(\gamma)}{2}\right)\)
	if \(\gamma\in S_\pm\) has  order \(r.\)   \end{thm}

\begin{proof}


The nontrivial stabilizer groups of the action of \(J\) on \(P\) are \(g\Gamma_\pm g^{-1}\) at \([g,0]\in G\times_{K_\pm}D^{d_\pm}\subset P.\)  
Since \(g\Gamma_\pm g^{-1}\) acts freely on \(V\bs\{0\}\) the nontrivial stabilizer  groups of \(J\) on \(P\times D^n\) are  \(g\Gamma_\pm g^{-1}\) at \(([g,0],0)\in (G\times_{K_\pm}D^{d_\pm})\times D^n.\)  Thus \(\bar I=\bar I_1\cup \bar I_-\cup \bar I_+\subset G\times P\times D^n\)  where 
\[\bar I_1=\{1\}\times P\times D^n\]
\[\bar I_\pm=\left\{(g\gamma g^{-1},([g,0],0))| \gamma\in\Gamma_\pm, g\in G\right\}.\]

We define a diffeomorphism 

\begin{equation}\label{icomps}S_\pm\times G/K_\pm^0\to \bar I_\pm\end{equation}
\[(\gamma,[g])\mapsto (g\gamma g^{-1},([g,0],0))\]
and note that the the diffeomorphism is equivariant with respect to the \(G\) action on the left of \(G/K_\pm^0\) and the action on \(\bar I\).  One checks easily that the map is well defined and surjective.  To check that it is injective, suppose 
\[(g\gamma g^{-1},([g,0],0))=(h\delta h^{-1},([h,0],0))\]
for \(\gamma,\delta\in S_\pm\) and \(g,h\in G.\)  Then \(h=gk\) for some \(k\in K_\pm\), and 
\[g\gamma g^{-1}=h\delta h^{-1}=gk\delta k^{-1}g^{-1}.\]
Thus \(\gamma\) and \(\delta\) are in the same conjugacy class in \(K_\pm;\) by the definition of \(S_\pm\), \(\gamma=\delta.\) Furthermore \(k\) commutes with \(\gamma\) so \(k\in K_\pm^0.\)  Thus 
\[(\delta,[h])=(\gamma,[gk])=(\gamma,[g]).\]    
We conclude that \(I_\pm=J\bs \bar I_\pm\) is diffeomorphic to \( S_\pm\times (J\bs G/K_\pm^0),\) proving the first part of the lemma.   

For each \(\gamma\in S_\pm\) let  \(\bar I_\gamma\subset \bar I_\pm\) be the image of  \(\{\gamma\}\times G/K_\pm^0\) under the diffeomorphism \eqref{icomps}.  These are the connected components of \(\bar I\bs \bar I_1\).  Let \(I_\gamma=J\bs \bar I_\gamma.\)    By equivariance \(C_{\gamma}\)is  the order of the ineffective kernel of \(Z_{\Gamma_\pm}(\gamma)=\Gamma_\pm\) on the horizontal subspace, with respect to \(J,\) of \(T_{(\gamma,[1])}(S_\pm\times G/K_{\pm}^0)\).   Indeed, since \(J\) is a normal subgroup of \(G,\) the adjoint action of \(J\) on \(\mathfrak{j}^\perp\subset \mathfrak{g}\) is trivial.  So \(\Gamma_\pm\subset J\) acts trivially on the horizontal subspace and  \(C_{\gamma}=\left|\Gamma_\pm\right|.\)

\(\overline{N}_\gamma\) is the orthogonal complement of the projection of \(T\bar I_\gamma\) to \(T(P\times D^n)\).  On \(\bar I_\pm\) that projection is fiberwise equal \(T((G\times_{K_\pm}\{0\})\times\{0\})\) and so \(\overline{N}_\gamma\) is fiberwise equal to the normal bundle of \((G\times_{K_\pm}\{0\})\times\{0\}\subset P\times D^n\) Since that is the zero section of a \(D^{d_\pm}\times D^n\) bundle, the normal bundle is \((G\times_{K_\pm}\ar{d_\pm})\times V.\)  \(\overline{N}_\pm=\coprod_{\gamma\in S_\pm}\overline N_{\gamma}\) is the pullback of that bundle to \(\bar I_\pm\):
\[\overline{N}_\pm=\left\{(g\gamma g^{-1},([g,x],v))|\gamma\in\Gamma_\pm, g\in G, x\in\ar{d_\pm}, v\in V\right\}.\] 

We define a bundle isomorphism, covering \eqref{icomps}, by
\[S_\pm\times \left(G\times_{K_\pm^0}(\ar{d_\pm}\times V)\right)\to \overline{N}_{\pm}\] 
\[(\gamma,[g,x,v])\to (g\gamma g^{-1},([g,x],gv)).\]

We note that the isomorphism is equivariant with respect to the left action of \(G\) on \(G\times_{K_\pm^0}(\ar{d_\pm}\times V)\) and the action on \(\overline{N}_\pm.\)  One checks easily that the map is well defined and surjective.  To check that it is injective, one need only check that is is fiberwise injective, which is apparent.

The action of \(dT\) corresponds under the diffeomorphism to the map 
\[(\gamma,[g,x,v])\mapsto (\gamma,[g,\gamma x,\gamma v]).\] It follows that for each \(\gamma\in S_\pm\)

\[\overline{N}_{\gamma}=\oplus_j \overline N_{\gamma,j}\cong\oplus_jG\times_{K_\pm^0}V_{\pm,j}\]
\(dT\) acts on \(\overline N_{\gamma,j}\) fiberwise as \(e^{i\theta_{\pm,j}(\gamma)}.\) The lemma follows from \lref{lgm}, up to a somewhat subtle discussion of orientation, which we undertake now.

We have seen that 
\[\text{Proj}_{TM}(T_{(g,p)}\bar I_-)=T_p((G\times_{K_-}\{0\})\times\{0\})\cong\mathfrak{k}_-^\perp\oplus\mathfrak{j}\] 
is the tangent space of the zero section of the \(D^{d_-}\times D^{n}\) bundle 
\[(G\times_{K_-}D^{d_-})\times D^n\subset P\times D^n\]
and \(\overline{N}_{(g,p)}\) can be identified with the \(D^{d_-}\times D^n\) fiber.  The proper orientation on \(\bar I_-\) will induce the orientation on \(P\times D^n\)	in conjunction with the orientation compatible with the complex structure on \(D^{d_-}\times D^n\subset \ar{d_-}\times V.\)  Away from the origin, \(T_xD^{d_-}=\text{span}(\partial r_-)\oplus\mathfrak{k}_-\) where \(r_-\) is the radial coordinate on \(D^{d_-}\) and \(\mathfrak{k}_-\) is identified by the action fields of \(K_-\) acting on \(\ar{d_-}.\)  Because of the quotient by \(K_-\) in the definition of the bundle, the action field of the \(K_-\) action on \(P\times D^n\) at \(([1,x],0)\) are the same as the image of the action fields on \(D^{d_-}.\)  Thus
\[\mathfrak{k}_\pm^\perp\oplus\mathfrak{j}\oplus\text{span}(\partial_{r_-})\oplus\mathfrak{k}_-\oplus T_0D^n\]
is opposite the orientation given on \(P\times D^n,\) which induced by 
\[\mathfrak{g}\oplus\text{span}(\partial_r)\oplus T_0D^n,\] noting the definition of the orientation on \(\mathfrak{k}_\pm^\perp\) given in the lemma, the fact that the dimension of \(\mathfrak{k}_-\), \(d_--1\), is odd, and \(\partial_r\) and the image of \(\partial_{r_-}\) are related by a positive constant.  If the two orientations on \(TD^{d_-}\) agree (the case \(\sigma_-=1\)) then the orientation on \(\bar I_-\) is the opposite of the one induced by \(\mathfrak{k}_\pm^\perp\oplus\mathfrak{j},\) and in turn the orientation induced on \(I_-\) is the opposite the one induced by \(\mathfrak{k}_-^\perp.\)  The diffeomorphism \eqref{icomps} is \(G\)- equivariant, so identifies the action fields  \(\mathfrak{k}_-^\perp\) in \(T_{[1]}J\bs G/K_-\) and \(TI_-\) , and \(-\sigma_-\) corrects the possible discrepancy between the orientation described  in the lemma and the induced orientation.

The analysis for \(I_+\) is similar, except if \(r_+\) is the radial coordinate on \(D^{d_+}\), \(\partial_{r_+}\) corresponds to \(-\partial_r\) up to scaling.  Thus the correction sign is \(\sigma_+.\)  
\
\end{proof}

\section{Highly Connected 7 Manifolds and Orbifold Bundles}\label{mfds}
	

In this section we describe the manifolds \(M_\du{a}{b}.\)  We will refer to the following subgroups of the unit quaternions \(S^3\subset\f{H}:\)
\[\Spin(2)=\{e^{i\theta}| \theta\in \ar{}\}\]
\[\Pin(2)=\Spin(2)\cup j\cdot\Spin(2)\]
\[Q=\{\pm1,\pm i,\pm j,\pm k\}.\]
Let \(\underline{a}=(a_1,a_2,a_3)\) and \(\underline{b}=(b_1,b_2,b_3)\) be triplets of integers such that
\begin{equation}\label{cond1} a_i,b_i=1\text{ mod }4\text{ for all }i\text{ \ \ \ \ and \ \ \ } \mathrm{gcd}(a_1,a_2,a_3)=\mathrm{gcd}(b_1,b_2,b_3)=1.\end{equation}  Goette, Kerin and Shankar define a 10-manifold \(P^{10}_{\underline{a},\underline{b}}\) using the  description in \sref{coho1prelim}  with groups 
	\[G=S^3\times S^3\times S^3\]
	\begin{equation}\label{groups}K_-=\{(e^{ia_1\theta},e^{ia_2\theta},e^{ia_3\theta})|\theta\in[0,2\pi]\}\cup \{(e^{ia_1\theta}j,e^{ia_2\theta}j,e^{ia_3\theta}j)|\theta\in[0,2\pi]\}\cong\Pin(2)\end{equation}  
    	\[K_+=\{(e^{jb_1\theta},e^{jb_2\theta},e^{jb_3\theta})|\theta\in[0,2\pi]\}\cup \{(e^{jb_1\theta}i,e^{jb_2\theta}i,e^{jb_3\theta}i)|\theta\in[0,2\pi]\}\cong\Pin(2)\]  
    	\[H=\{(q,q,q)|q\in Q\} .\]
    	
    \
    
    \noindent 	By \tref{GZ}, \(P^{10}_{\underline{a},\underline{b}}\) admits a \(G\) invariant metric of nonnegative sectional curvature.	If we assume further that

    \begin{equation}\label{cond2}\mathrm{gcd}(a_1,a_2\pm a_3)=\mathrm{gcd}(b_1,b_2\pm b_3)=1\end{equation}  then the subgroup
    	\[1\times \Delta S^3=\{(1,g,g)|g\in S^3\}\subset S^3\times S^3\times S^3\]
    	acts freely on \(P^{10}_{\underline{a},\underline{b}}\) and \(M^7_{\underline{a},\underline{b}}=(1\times\Delta S^3)\backslash P^{10}_{\underline{a},\underline{b}}\) is a 2-connected manifold with \(\zco{M_{\underline{a},\underline{b}}}{4}=\f{Z}_{|n({\underline{a},\underline{b}})|}\), where 
    	\begin{equation}\label{ndef}n({\underline{a},\underline{b}})=\ov{8}\det\mat{a_1^2}{b_1^2}{a_2^2-a_3^2}{b_2^2-b_3^2}.\end{equation}  See \cite {GKS} Section 2 for details.  It is immediate from the O'Neil formula that \(M^7_{\underline{a},\underline{b}}\) admits a metric of nonnegative sectional curvature.

    	

    	In order to use the formulae of \sref{ks}, we want to identify \(M_{\underline{a},\underline{b}}\) with the boundary of an orbifold disc bundle and define appropriate metrics.   
    	We inspect the action of \(J=1\times S^3\times S^3\) on \(P_{\underline{a},\underline{b}}.\)  Observing which elements of \(J\) are conjugate to elements of \(K_\pm\) or \(H\) we have
    	
    	\begin{lem}\label{genstab}
    			A point \(p\in P_{\underline{a},\underline{b}}\)  has nontrivial stabilizer in \(1\times S^3\times S^3\)  if and only if  
    		\[p=[g,0]\in G\times_{K_\pm}\{0\}\subset G\times_{K_\pm}D^2.\]
    		The stabilizer of such a point \(p\) is \(g\f{Z}_{|a_1|}g^{-1}\) in the case of \(K_-\)  or  \(g\f{Z}_{|b_1|}g^{-1}\) in the case of \(K_+\), where \(\f{Z}_{|a_1|}\subset J\) is generated by \((1,e^{2\pi i a_2/|a_1|},e^{2\pi i a_3/|a_1|})\) and \(\f{Z}_{|b_1|}\subset J\) is generated by \((1,e^{2\pi i b_2/|b_1|},e^{2\pi i b_3/|b_1|})\)
    		\end{lem}

    	 
    	 
    	 
    	 It follows that \[B_{\underline{a},\underline{b}}={J}\backslash P_{\underline{a},\underline{b}}\] is a 4-dimensional orbifold and \(P_{\underline{a},\underline{b}}\) is an orbifold principal \(S^3\times S^3\)  bundle over \(B_{\underline{a},\underline{b}}.\)  Let
    	 
    	  \[E_{\underline{a},\underline{b}}=J\bs (P_{\underline{a},\underline{b}}\times\ar{4})\] 
    	  \[W_{\underline{a},\underline{b}}=J\bs (P_{\underline{a},\underline{b}}\times D^4)\]
    	  \[\partial W_{\underline{a},\underline{b}}=J\bs( P_{\underline{a},\underline{b}}\times S^3)\]
  
  \
   
   \noindent 	  be the associated orbifold vector, disc, and sphere bundles.  Here \(J\) acts on \(\ar{4}\cong\f{H}\)  by quaternion multiplication 
    	 \((1,g,h)\cdot v=gvh^{-1}.\)  One checks that the map 
    	 \begin{align*}P_{\underline{a},\underline{b}}\times_{J}S^3&\to (1\times \Delta S^3)\backslash P_{\underline{a},\underline{b}}=M_{\underline{a},\underline{b}}\\
    	 [p,g]&\mapsto[(1,1,g)p]\end{align*}
    	is a diffeomorphism.  
    	
    	We have the bundle structure required to apply \lref{sbound} and \lref{sbase}.   In the following lemma we define the appropriate metrics.
    	
    	\begin{lem}\label{metrics}
    		\(W_{\underline{a},\underline{b}}\) admits a product-like metric \(g_{W_{\underline{a},\underline{b}}}\) such that {\normalfont sec}\((g_{W_{\underline{a},\underline{b}}})\geq0\) and\\ {\normalfont scal}\((g_{W_{\underline{a},\underline{b}}})>0\).  
    	\end{lem}

    \begin{proof}
    	By \tref{GZ}, \(P_{\underline{a},\underline{b}}\) admits a \(G\) invariant metric \(g_{P_{\underline{a},\underline{b}}}\) with sec\((g_{P_{\underline{a},\underline{b}}})\geq0.\)  Let \(g_{D^4}\) be an \SO(4)  invariant metric on \(D^4\) with sec\((g_{D^4})\geq 0\) which is a product with a round metric near \(\partial D^4.\)  Define \(g_{W_{\underline{a},\underline{b}}}\)  such that \((P_{\underline{a},\underline{b}}\times D^4,g_{P_k}+g_{D^4})\to(W_{\underline{a},\underline{b}},g_{W_{\underline{a},\underline{b}}})\) is a Riemannian submersion.  Then \(g_{W_{\underline{a},\underline{b}}}\) is product like near \(\partial W_{\underline{a},\underline{b}}\), and by the O'Neil formula sec\((g_{W_{\underline{a},\underline{b}}})\geq0.\)  Furthermore, \(g_{W_{\underline{a},\underline{b}}}\) has a positively curved plane at each point, so scal\((g_{W_{\underline{a},\underline{b}}})>0;\)  see \cite{G} Theorem 2.1 for details.  The O'Neil formula applies locally on each orbifold chart, and thus the computations are identical to the manifold case.  
    \end{proof}

    Define \(g_{M_\du{a}{b}}=g_{W_{\underline{a},\underline{b}}}|_{M_{\underline{a},\underline{b}}}\).  Since \(g_{W_\du{a}{b}}\) is product like, \(\sec{g_{M_\du{a}{b}}}\geq 0\) and scal\((g_{M_\du{a}{b}})>0\).  Note that the metric \(g_{M_\du{a}{b}}\) differs from the metric, denoted \(g'_{M_{\underline{a},\underline{b}}}\),    which makes \((P_{\underline{a},\underline{b}},g_{P_{\underline{a},\underline{b}}})\to (M_{\underline{a},\underline{b}},g'_{M_{\underline{a},\underline{b}}})\) a Riemannian submersion; rather we have a Riemannian submersion   \((P_{\underline{a},\underline{b}}\times S^3,g_{P_{\underline{a},\underline{b}}}+g_{S^3})\to (M_{\underline{a},\underline{b}},g_{M_\du{a}{b}}),\)  where \(g_{S^3}\) is a round metric.  However, by increasing the radius of \(g_{S^3}\) towards \(\infty\) (a Cheeger deformation) , we obtain a continuous path of metrics with nonnegative curvature from \(g_{M_\du{a}{b}}\) to \(g'_{M_{\underline{a},\underline{b}}}\).

    
    \begin{lem}\label{swosint}    
   \begin{equation}\label{sform}s(M_{\underline{a},\underline{b}},g)=-\frac{|n({\underline{a},\underline{b}})|-{a_1^2b_1^2m({\underline{a},\underline{b}})^2}}{2^5\cdot 7n({\underline{a},\underline{b}})}-D(a_1;4,a_3+a_2,a_3-a_2)+D(b_1;4,b_3+b_2,b_3-b_2)\end{equation}
   where 
   
   \[m(\underline{a},\underline{b})=\ov{8a_1^2b_1^2}\det\mat{a_1^2}{b_1^2}{a_2^2+a_3^2+8}{b_2^2+b_3^2+8}\]
   and 	\[D(q;p_1,p_2,p_3)=\ov{2^5\cdot 7q^2}\sum_{l=1}^{(|q|-1)/2}\sum p_i\left(\frac{14\cos\left(\frac{p_i\pi l}{q}\right)+\cos\left(\frac{p_j\pi l}{q}\right)\cos\left(\frac{p_k\pi l}{q}\right)}{\sin^2\left(\frac{p_i\pi l}{q}\right)\sin\left(\frac{p_j\pi l}{q}\right)\sin\left(\frac{p_k\pi l}{q}\right)}\right)\]
    The second summation above is over each \((i,j,k) \in\{(1,2,3),(2,3,1),(3,1,2)\}.\)
   \end{lem}
        	We note that up to a sign, which indicates a change in orientation, the right side of \eqref{sform} is equal to the right side of \cite[Thm. C]{GKS}, confirming \eqref{smu}.   
        	
        	First we set some notation.  Let \(C_w\) be the complex 1-dimensional representation of \(S^1\) of weight \(w\in\f{Z}\) (that is, \(z\in S^1\) acts on \(C_w=\f{C}\) by \(z^w\)).  Let \(E_{a,b}\to S^2\) be the orbifold complex line bundle over the manifold \(S^2\) defined by \(E_{a,b}=S^3\times_{S^1}C_b\to S^3/S^1,\) where \(z\in S^1\) acts on the right of \(S^3\) as \(z^{-a}.\)  Note \(S^1\) acts on \(S^3\) with innefective kernel \(\f{Z}_{|a|},\) but \(S^1/\f{Z}_{|a|}\) acts freely and thus \(S^3/S^1=S^2\) is a manifold.  We orient \(S^2\) such that the images in \(T_{[1]}S^3/S^1\) of \(j,k\) in the Lie algebra \(\mathfrak{s}^3\) of \(S^3\) form an oriented basis. Then

        	\begin{lem}\label{s2chern}
        		\[	\int_{S^2}c_1(E_{a,b})=-\frac{b}{a}.\]
        	\end{lem} 
        	
        	\begin{proof}
        		The principal \(S^1\) orbibundle of \(E_{a,b}\) is 
        		\(S^3\times_{S^1} S^1\subset E_{a,b}.\)  The diffeomorphism 
        			\[S^3\times_{S^1} S^1\to S^3/\f{Z}_{|b|},\]
        			\[[g,z]\mapsto[gz^{a/b}]\]
        			commutes with the projection to \(S^3/S^1\cong S^2\) and is equivariant with respect to the \(S^1\) action on the right of \(S^3\times_{S^1}S^1\) and the action \([g]\cdot z=[gz^{a/b}]\) of \(z\in S^1\) on the right of \(S^3/\f{Z}_{|b|}.\)  
        		
        		Give \(S^3/\f{Z}_{|b|}\) a metric locally isometric to the round  metric of radius 1 on \(S^3\) and give \(S^2\) the round metric with radius \(1/2.\)  Then \(S^3/\f{Z}_{|b|}\to S^2\) is a Riemannian submersion.  The metric on \(S^3/\f{Z}_{|b|}\) defines a principal connection for the \(S^1\) orbibundle.  An oriented orthonormal basis of \(T_{[1]}S^2\) lifts to \(\{j,k\}\subset T_{[1]}S^3/\f{Z}_{|b|}\cong T_1S^3.\) The action field of the \(S^1\) action corresponding to \(i\) in the Lie algebra \( \mathfrak{s}^1=\mathfrak{u}(1)\), which is left invariant, is given at \([1]\) is given by 
        		\[\left.\frac{d}{dt}\right|_{t=0}[e^{iat/b}]=\frac{a}{b}i.\] 
        		Thus the connection form \(\theta\in\Omega(S^3/\f{Z}_{|b|},\mathfrak{u}(1))\) is given by \[\theta(X)=\frac{\left<X,\frac{a}{b}i\right>}{\left|\frac abi\right|^2}i\] the curvature of the principal connection evaluated on the orthonormal basis of \(T_{[1]}S^2\) is
        		\[d\theta(j,k)={\left<-[j,k],\frac{b}{a}i\right>}i=-\frac{2b}{a}i\]   
        		
        		Since the metric on \(S^3/\f{Z}_{|b|}\) is homogeneous, the connection and curvature are homogenous.  Thus curvature  form on \(S^2\) is 
        		\[R=-\frac{2b}{a}id\text{vol}_{S^2}\] and the first Chern form is 
        		\[\frac{R}{2\pi i}=-\frac{b}{a\pi}d\text{vol}_{S^2}.\]
        		Thus 
        		\[\int_{S^2}c_1(E_{a,b})=-\frac{b}{a\pi}\text{vol}(S^2)=-\frac{b}{a}.\]

        	\end{proof}
        	
        
        \begin{proof}[Proof of \lref{swosint}]
  To see that \(W_{\underline{a},\underline{b}}\) is spin, note that \(H^2(P_{\underline{a},\underline{b}},\f{Z}_2)=0\) and thus \(P_k\) is spin (this follows for instance from the Gysin sequence since \(\twoco{M_{\underline{a},\underline{b}}}{2}=0\) ).  The action of \(J\) splits \(T(P_{\underline{a},\underline{b}}\times D^4)=v\oplus h\) into a vertical bundle \(v\) and horizontal bundle \(h\).  Since \(J\) acts almost freely \(v\) is trivializable and spin.  It follows that \(h\) is spin. \(J\) acts on the frame bundle \(\SO(h)\) of \(h\) with quotient \(\SO(TW_{\underline{a},\underline{b}})\).  Since \(1\times S^3\times S^3\) is simply connected, that action lifts to an action on \(\Spin(h)\).  The quotient \(J\backslash \Spin(h)\) is a spin structure for \(W_{\underline{a},\underline{b}}.\)    
  
  Since \(\zco{M_{\underline{a},\underline{b}}}{4}\) is torsion, \(p_1(TM_{\underline{a},\underline{b}})=0\in H^4(M_\du{a}{b},\ar{})\) We can thus apply \lref{sbound} to \(M_{\underline{a},\underline{b}}=\partial W_{\underline{a},\underline{b}}\) with the metric \(g_{W_{\underline{a},\underline{b}}}\) defined in \lref{metrics} and \(g_{M_{\underline{a},\underline{b}}}=g_{W_{\underline{a},\underline{b}}}|_{\partial W_{\underline{a},\underline{b}}}.\)   Since \(W_{\underline{a},\underline{b}}\subset E_{\underline{a},\underline{b}}\to B_{\underline{a},\underline{b}}\) is an orbifold disc bundle we can apply \lref{sbase}. In \cite{GKS} Lemma 7.9 and 7.11 the authors integrate curvature forms of explicit connections to compute 
         \[\int_Be(E)=-\frac{n}{a_1^2b_1^2}\]
        \begin{equation}\label{generalbaseint}\int_Bp_1(TB)=-\ov{4a_1^2b_1^2}\det\mat{a_1^2}{b_1^2}{8}{8}\end{equation}
        \[\int_Bp_1(E)=-\ov{4a_1^2b_1^2}\det\mat{a_1^2}{b_1^2}{a_2^2+a_3^2}{b_2^2+b_3^2}\]
        so \(\int_Bp_1(TB)+\int_Bp_1(E)=-2m.\)  Note the signs are opposite those give in \cite{GKS}; we use the opposite orientation on \(B,\) for reasons described below.  Applying \lref{sbound} and \lref{sbase}, we have 
        \[s(M_{\underline{a},\underline{b}},g_\du{a}{b})=-\frac{|n({\underline{a},\underline{b}})|-{a_1^2b_1^2m({\underline{a},\underline{b}})^2}}{2^5\cdot 7n({\underline{a},\underline{b}})}-\int_{\Lambda_sW_{\underline{a},\underline{b}}}\left(\hat{A}_\Lambda+\ov{2^5\cdot 7}L_\Lambda\right)\]
	We  write \(W=W_{a,b}\) in the language of \tref{coho1} with\(J=1\times S^3\times S^3\) and \(V=\f{H}\) such that \(G\) acts by \((g_1,g_2,g_3)\cdot x=g_2xg_3^{-1}\).  Then  \(\Gamma_-=\f{Z}_{|a_1|}\subset \text{Pin(2)}\cong K_-\), and \(\Gamma_+=\f{Z}_{|b_1|}\subset \text{Pin(2)}\cong K_+.\)    that gcd\((a_1,a_2\pm a_3)=1\) and  gcd\((b_1,b_2\pm b_3)=1\) imply that \(\Gamma_\pm \) acts freely on \(V\bs\{0\}\) ( indeed, the gcd assumptions we made so that \(J\) would act freely on \(P_{\underline{a},\underline{b}}\times S^3\)).  The orders of \(\Gamma_\pm\) are odd and each nontrivial element has centralizer \(\Spin(2)\cong K_\pm^0\) in \(\Pin(2)\cong K_\pm\).

		The \(\Pin(2)\) conjugacy class of an element \(\gamma\in \Spin(2)\) is \(\{\gamma,\gamma^{-1}\}.\) We choose 
		\[S_-=\left\{\lambda^k| k=1,...,\frac{|a_1|-1}{2}\right\},\quad \lambda=e^{2\pi i/|a_1|}\]
		\[S_+=\left\{\mu^k|k=1,...,\frac{|b_1|-1}{2}\right\},\quad \mu=e^{2\pi i/|b_1|}.\]
		\(z\in Spin(2)\cong K_-^0\) acts on \(V=\f{H}\) by \(z^{a_2}\) on the left and \(z^{-a_3}\) on the right.   Under the  identification \(\f{H}=\f{C}^2\)  this is equal to the action of \(\diag(z^{a_2-a_3},z^{a_2+a_3})\). The diffeomorphism 
		\(K_\pm/H\cong \Pin(2)/Q\cong S^1\)
		is given by \([e^{i\theta}]\mapsto e^{4i\theta}\) and so \(z\in\Spin(2)\cong K_\pm^0\) acts on \(\ar{d_-}=\ar{d_+}=\ar{2}\cong \C\) by \(z^4.\)   
		
		Thus as representations of \(K_-^0\),
		\[\ar{2}\oplus V=V_{-,1}\oplus V_{-,2}\oplus V_{-,3}\]
		\[V_{-,1}=C_4, V_{-,2}=C_{a_2-a_3}, V_{-,3}=C_{a_2+a_3}\]
		with 
		\begin{equation}\label{-angles}\theta_{-,1}(\lambda^k)=\frac{8\pi k}{|a_1|}\, \theta_{-,2}(\lambda^k)=\frac{2(a_2-a_3)\pi k}{|a_1|}, \theta_{-,3}(\lambda^k)=\frac{2(a_2+a_3)\pi k}{|a_1|}\end{equation}
		
		By a similar argument, as representations of \(K_+^0\),
		\[\ar{2}\oplus V=V_{+,1}\oplus V_{+,2}\oplus V_{+,3}\]
		\[V_{+,1}=C_4, V_{+,2}=C_{b_2-b_3}, V_{+,3}=C_{b_2+b_3}\]
		with 
		\begin{equation}\label{+angles}\theta_{+,1}(\mu^k)=\frac{8\pi k}{|b_1|}\, \theta_{+,2}(\mu^k)=\frac{2(b_2-b_3)\pi k}{|b_1|}, \theta_{+,3}(\mu^k)=\frac{2(b_2+b_3)\pi k}{|b_1|}.\end{equation}

		Note that the angles given are not necessarily in to be in the range \((0,\pi)\) stipulated in the formulae of \lref{lgm} and \lref{coho1}.  This corresponds to the opposite choice of complex structure, which can be seen to lead to a possible sign change in the formulae; however, one checks that the possible sign change is compensated by the corresponding change in orientation.  
		
		The complex structure on \(C_4\) matches the orientation induced by \(\partial r\) and the action fields of the  \(S^1\) action, so \(\sigma_+=\sigma_-=1.\)  
		We can then conclude that since each \(a_i\) and \(b_i\) is odd, and the orders of \(\lambda^k\) and \(\mu^k\) divide \(|a_1|\) and \(|b_1|\) respectively,
		\[\ep(\lambda^k)=\cos(4\pi k)\cos((a_2-a_2)\pi k)\cos((a_2+a_3)\pi k)=1\]
		and 
		\[\ep(\mu^k)=\cos(4\pi k)\cos((b_2-b_2)\pi k)\cos((b_2+b_3)\pi k)=1.\]
		
		We see that 
		\[J\bs G/K_{\pm}^0=S^3/S^1=S^2\] where \(S^1\) acts on \(F=S^3\) with ineffective kernel \(\Gamma_\pm,\) but \(S^1/\Gamma_\pm\) acts freely.  Furthermore 
		\begin{align}\label{bundlepieces}F\times_{K_-^0}V_{-,1}&=S^3\times_{S^1}C_4=E_{a_1,4} & F\times_{K_+^0}V_{+,1}&=E_{b_1,4}\\
		F\times_{K_-^0}V_{-,2}&=E_{a_1,a_2-a_3} &	F\times_{K_+^0}V_{+,2}&=E_{b_1,b_2-b_3}\\
		F\times_{K_-^0}V_{-,3}&=E_{a_1,a_2+a_3} &
		F\times_{K_+^0}V_{+,3}&=E_{b_1,b_2+b_3}\end{align}

		Next we determine the proper orientation on \(W\) and \(J\bs G/K_\pm^0\) to apply apply \lref{sbase} and \lref{coho1}.  As described before \lref{coho1},  we give \(P_{\underline{a},\underline{b}}\times D^4\) the orientation  induced by the identification away from the singular orbits with 
		\[G/H\times(-1,1)\times D^4,\] where the orientation on \(D^4\) is induced by the complex structure on \(V=\f{H}\).  \(W\) inherits an orientation as the quotient by \(J.\)  To apply \lref{sbase}, we orient \(B_\du{a}{b}\) such that \(W\) is oriented as the total space of an oriented \(\f{H}=\ar{4}\) bundle with the standard orientation on the fibers.  That orientation on \(B_\du{a}{b}\) will correspond to the identification with 
		\[J\bs G/H\times (-1,1)=S^3/Q\times (-1,1)\]
		away from the singular locus.  Thus the orientation on \(B_\du{a}{b}\) is induced by \(\mathfrak{s}^3\oplus\text{span}\{\partial_r\},\) where \(r\) is the  coordinate on \(-1,1\).  This is opposite to the orientation used in \cite{GKS}, leading to the sign difference in \eqref{generalbaseint}.  
		
		We orient \(J\bs G/K_\pm^0\)  using the description before \tref{coho1}.   \(\mathfrak{s}^3\)  is oriented by the standard basis \(\{i,j,k\}\) and we have

		\begin{align*}\mathfrak{g}&=\mathfrak{s}^3\oplus \mathfrak{s}^3\oplus \mathfrak{s}^3 & \mathfrak{k}_-&=\text{span}\{(a_1i,a_2i,a_3i)\}
		\\\mathfrak{j}&=1\oplus\mathfrak{s}^3\oplus \mathfrak{s}^3 & \mathfrak{k}_+&=\text{span}\{(b_1i,b_2i,b_3i)\}\end{align*}
		We see that \((j,0,0), (k,0,0)\) form an oriented basis of \((\mathfrak{j}+\mathfrak{k_-})^\perp\), as described before  \tref{coho1}, if and only if \(a_1>0.\)  Similarly \((j,0,0), (k,0,0)\) form an oriented basis of \((\mathfrak{j}+\mathfrak{k_+})^\perp\) if and only if \(b_1>0.\)  Since the basis  \((j,0,0), (k,0,0)\) induces the orientation on \(S^2\) used in \lref{s2chern}, we adjust that computation with the requisite sign. 
		
		Noting \(L(TS^2)=\hat{A}(TS^2)=1,\) we expand the \(\coth\) and csch terms in \tref{coho1} to first order in \(x_{\pm,j,k}\), and substitute the angles in \eqref{-angles} and \eqref{+angles} and the Chern classes if the bundles in \eqref{bundlepieces}, and integrate using \lref{s2chern} to find

		\[-\int_{\Lambda_sW_{\underline{a},\underline{b}}}\left(\hat{A}+\ov{2^5\cdot 7}L_\Lambda\right)=D(|a_1|;4,a_2-a_3,a_2+a_3)-D(|b_1|;4,b_2-b_3,b_2+b_3)\]

		The lemma follow by noting that \(D(q;p_1,p_2,p_3)\) is independent of the sign of \(q\) and permutations of the \(p_i,\) and changes sign when  \(p_i\) does.    
	\end{proof}

	\section{Diffeomorphism Invariants and Proof}\label{diffinv}
	In order to prove \tref{odd} and \tref{thmB} we identify diffeomorphisms between the manifolds \(M_{\underline{a},\underline{b}}\).  We use the classification of Crowley \cite{Cr}, described in \sref{exotic}, in terms of the invariants \(\mu\) and \(q\).  
	
	 \(\mu(M_\du{a}{b})=s(M_\du{a}{b},g)\) mod \(\f{Z}\) is determined by \eqref{sform}, as computed in \cite[Theorem C]{GKS}.  In \cite[Theorem B]{GKSl}, the authors show that there is a generator \({\bf1}_{\du{a}{b}}\in\zco{M_{\underline{a},\underline{b}}}{4}\) such that the linking form is determined by
	\begin{equation}\label{linkingform}\text{lk}({\bf1}_\du{a}{b},{\bf1}_\du{a}{b})=\ov{n}\left(e_1b_1^2+e_0\left(\frac{b_2^2-b_3^2}{8}\right)\right)\text{ mod }\f{Z}\end{equation}
	where \(e_0,e_1\) are integers such that \(e_1a_1^2+e_0\left(\frac{a_2^2-a_3^2}{8}\right)=1.\)  
	
	Recall from \sref{exotic} that if \(H^4(M,\f{Z})\) has odd order and \(\mu(M)=\mu(N),\) then the existence of an isomorphism from \(\zco{M}{4}\to \zco{N}{4}\) preserving lk and \(p_1\) guarantees that \(M\) and \(N\) are diffeomorphic.  In \lref{pontlem} below we determine that if \(a_1\) and \(b_1\) are relatively prime, 
	\[p_1(TM_\du{a}{b})=\pm2a_1^2m(\du{a}{b}){\bf 1}_\du{a}{b}.\]
	\(m(\du{a}{b})\) is defined in the statement of \lref{swosint}.  
	
	\

\begin{proof}[Proof of Theorem A and Theorem C]		Let \(\underline{a}=(a_1,a_2,a_3)\) and \(\underline{b}=(b_1,b_2,b_3)\) satisfy \eqref{cond1} and \eqref{cond2} such that \(n(\du{a}{b})\neq 0,\) see \eqref{ndef}.  For each \(i\in\f{Z}\) define \[\underline{a}_i=(a_{i,1},a_{i,2},a_{i,3})=(a_1,a_2+a_1^2(b_3-b_2)i,a_3+a_1^2(b_3-b_2)i)\]
		\[\underline{b}_i=(b_{i,1},b_{i,2},b_{i,3})=(b_1,b_2+b_1^2(a_3-a_2)i,b_3+b_1^2(a_3-a_2)i)\]
		
		 One checks that \(n(\underline{a},\underline{b})=n(\underline{a}^i,\underline{b}^i)\).  	  Let \(e_{1,i}=e_1-\ov{4}i(b_3-b_2)(a_2-a_3).\)  
		 Then 
		 \[e_{1,i}a_1^2+e_0\left(\frac{(a_2+a_1^2(b_3-b_2)i)^2-(a_3+a_1^2(b_3-b_2))^2}{8}\right)=1\] and so 
		 \begin{align*}\mathrm{lk}(1_\du{a_i}{b_i},1_\du{a_i}{b_i})&=e_{1,i}b_1^2+e_0\left(\frac{(b_2+b_1^2(a_3-a_2)i)^2-(b_3+b_1^2(a_3-a_2)i)^2}{8}\right)\\&=e_1b_1^2+e_0\left(\frac{b_2^2-b_3^2}{8}\right)=\mathrm{lk}(1_\du{a}{b},1_\du{a}{b}).\end{align*}
		 Thus the isomorphism  \(\phi_i:\zco{M_\du{a}{b}}{4}\to\zco{M_\du{a_i}{b_i}}{4}\) which sends \(1_\du{a}{b}\) to \(1_\du{a_i}{b_i}\)  preserves  the linking form.
		 
		   Since \(D(q;p_1,p_2,p_3)\) is periodic in \(p_3\) with period \(2q\) (see \lref{swosint}), we see that 
		 \[D(a_1^i,4,a_2^i-a_3^i,a_2^i+a_3^i)=D(a_1,4,a_2-a_3,a_2+a_3+2a_1^2(b_3-b_2)i)=D(a_1,4,a_2-a_3,a_2+a_3)\] \[D(b_1^i,4,b_2^i-b_3^i,b_2^i+b_3^i)=D(b_1,4,b_2-b_3,b_2+b_3+2b_1^2(a_3-a_2)i)=D(b_1,4,b_2-b_3,b_2+b_3).\]  Thus the only term in \(s(M_{\underline{a}^i,\underline{b}^i},g_{\du{a_i}{b_i}})\) which depends on \(i\) is \[\frac{a_1^2b_1^2}{2^5\cdot7n(\underline{a},\underline{b})}m(\underline{a}^i,\underline{b}^i)^2=\frac{1}{2^5\cdot7n(\du{a}{b})a_1^2b_1^2}\left(A+Bi+Ci^2\right)^2\]
		 where  \(A,B,C\) are integers independent of \(i\) given by 	 
		 \[A=a_1^2b_1^2m(\du{a}{b})\]\[ B=\ov{4}\det\mat{a_1^2}{b_1^2}{a_1^2(a_2+a_3)(b_3-b_2)}{b_1^2(b_2+b_3)(a_3-a_2)}\]
		 \[C=\ov{4}\det\mat{a_1^2}{b_1^2}{a_1^4(b_3-b_2)^2}{b_1^4(a_3-a_2)^2}.\]
		 Thus if \(i=0\) mod \(2^5\cdot 7n(\du{a}{b})a_1^2b_1^2\), \(\mu(M_\du{a_i}{b_i})=\mu(M_\du{a}{b}).\)	One checks that if \(B=C=0\),  then \(n(\underline{a},\underline{b})=0.\)  We conclude that  \(s(M_{\underline{a}^i,\underline{b}^i},g_{\du{a_i}{b_i}})\) is a nontrivial polynomial in \(i.\)  
		 
		 The manifolds \(M_\du{a_i}{b_i}, i\in2^5\cdot 7n(\du{a}{b})a_1^2b_1^2\f{Z}\) may fall into multiple diffeomorphism types.  However, since those manifolds have the same value of \(\mu\) and the same order of \(H^4\), and there are finitely many possibilities for \(q\) for a given finite group, they fall into finitely many diffeomorphism types. Thus there is an infinite set \(S\subset 2^5\cdot 7n(\du{a}{b})a_1^2b_1^2\f{Z}\) such that the manifolds \(M_\du{a_i}{b_i}, i\in S\) are all diffeomorphic.  But since \(s\) is given by a nontrivial polynomial in \(i\),  \(\{|s(M_\du{a_i}{b_i},g_\du{a_i}{b_i})||i\in S\}\) is infinite.   Since \(|s|\) is preserved when we pull back metrics between simply connected manifolds (see \sref{ksinv}), the  pullbacks of the metrics \(\{g_\du{a_i}{b_i}|i\in S\}\) to any \(M_\du{a_j}{b_j}\), \(j\in S\)  obtain an infinite set of values of \(|s|\), and by \lref{comps}  represents an infinite set of components of \(\modse{sec}(M_\du{a_j}{b_j})\).  This completes the proof of \tref{thmB}.

		 If \(n(\du{a}{b})\) is odd and gcd\((a_1,b_1)=1\), by \lref{pontlem}, replacing \(1_\du{a_i}{b_i}\) with \(-1_\du{a_i}{b_i}\) if necessary, we may assume \(p_1(TM_\du{a_i}{b_i})= 2a_1^2m(\du{a_i}{b_i})1_\du{a_i}{b_i}.\)  Note changing the sign of \(1_\du{a_i}{b_i}\) does not affect \eqref{linkingform}, and thus the isomorphisms \(\phi_i\) described above still preserve the linking form.   
		 
		 Then 
		 \[p_1\left(TM_\du{a_i}{b_i}\right)=\frac2{b_1^2}\left(A+Bi+Ci^2\right){\bf1}_\du{a_i}{b_i}\]	
		  and  if \(i=0\) mod \(n(\du{a}{b})\), \(\phi_i\) maps \(p_1\left(TM_\du{a}{b}\right)\) to  \(p_1\left(TM_\du{a_i}{b_i}\right).\)  Note that \(b_1\) and \(n(\du{a}{b})\) are relatively prime by assumption and \eqref{cond2}.  
		  
		Since \(n(\du{a}{b})\) is odd, the isomorphisms \(\phi_i\) and the values of \(\mu\) demonstrate that the manifolds \(M_\du{a_i}{b_i}, i\in2^5\cdot 7n(\du{a}{b})a_1^2b_1^2\f{Z}\) are all diffeomorphic, and thus all diffeomorphic to \(M_\du{a}{b}\) \((i=0)\).  We  repeat the analysis above with \(S=2^5\cdot 7n(\du{a}{b})a_1^2b_1^2\f{Z}\) and conclude that \(\modse{sec}(M_\du{a}{b})\) has infinitely many components, proving \tref{odd}.    \end{proof}

		To prove \cref{thm} we apply \tref{odd}  to the family
		\begin{equation}M_k=M_{(-3,-3,1),(1,4k+1,4k+1)}.\end{equation}	Since \(|n((-3,-3,1),(1,4k+1,4k+1))|=1\), \(M_k\) is homeomorphic to \(S^7.\)
		In \cite[Corollary D]{GKS}, the authors  compute 
		\begin{equation}\mu(M_k)=\frac{9}{56}(4k^4+4k^3+3k^2+k)\text{ mod }\f{Z}\end{equation}
		and observe that the manifolds
		\[M_{-3},M_{-1},M_{1},M_{2},M_{4},M_{8},M_{11},M_{15}\]
		represent every diffeomorphism type of those homotopy spheres which are not Milnor spheres, see \eqref{nm}.  The case of the Milnor spheres is demonstrated in \cite{D} and \cite{G} and also follows from \tref{odd}, using the observation from \cite{GZ} that every Milnor sphere is diffeomorphic to \(M_{(1,q_-,p_-),(1,q_+,p_+)}\) for proper values of \(q_\pm,p_\pm.\)


	\section{Pontryagin Class}\label{pontryagin}
	In this section we compute the first Pontryagin class of \(M_\du{a}{b}\) in the case where \(a_1\) and \(b_1\) are relatively prime and \(n(\du{a}{b})\) is odd.  The computation was used in the proofs in \sref{diffinv}. 	Let \(\du{a}{b}\) be triplets of integers satisfying \eqref{cond1} and \eqref{cond2}.  From the definition of \(P_\du{a}{b}\) in \sref{mfds} (see \sref{coho1prelim} and the proof of \lref{swosint}) the orbifold \(B_\du{a}{b}=1\times S^3\times S^3\bs P_\du{a}{b}\) can be described as 
	\begin{equation}\label{bdecomp}B_\du{a}{b}=S^3\times_{K_-}D^2\cup_{S^3/Q} S^3\times_{K_+}D^2\end{equation}
	Here \(e^{i\theta}\in K_- \cong \text{Pin}(2)\) acts on the right of \(S^3\) as \(e^{ia_1\theta}\) and on the left of \(D^2\) as \(e^{4i\theta},\) while \(j\) acts of the left of \(S^3\) as itself and trivially on \(D^2.\)  \(K_+\) acts similarly with \(a_1\) replaced by \(b_1\) and the roles of \(i\) and \(j\) reversed.  It follows that \(B_{\du{a}{b}}\) depends only on \(a_1\) and \(b_1,\) not on the entire triplets \(\underline{a}\) and \(\underline{b}.\)  Accordingly, we will use the notation \(B_{a_1,b_1}:=B_\du{a}{b}\) it what follows.  
	
	Let \(EJ\) be the contractible total space of the universal principal \(S^3\times S^3\) bundle.   We can represent the classifying space of the orbifold \(B_{a_1,b_1}\) (see \sref{orbiclasses}) as
	\[BB_{a_1,b_1}=J\bs (P_{\du{a}{b}}\times EJ)\]
	and of \(M_{\du{a}{b}}\) by 
	\[BM_\du{a}{b}={1\times \Delta S^3}\bs (P_{\du{a}{b}}\times EJ).\]
	 \(1\times \Delta S^3\) acts freely on \(P_\du{a}{b}\), so  \(BM_\du{a}{b}\) is homotopy equivalent to \(M_\du{a}{b}\), as we expect for a manifold.  We have an orbifold \(S^3\) bundle
	 \[S^3\to M_\du{a}{b}\xrightarrow{\pi} B_{a_1,b_1}\]
	 and a (standard) orientable \(S^3\) bundle 
	\begin{equation}\label{s3bundle}S^3\to BM_\du{a}{b}\xrightarrow{B\pi} BB_{a_1,b_1}.\end{equation}
Noting that \(P_\du{a}{b}\) and \(EJ\), and therefore \(BB_{a_1,b_1},\) are simply connected,  the Gysin sequence for that \(S^3\) bundle yields the exact sequence
\begin{equation}\label{gysin}0\to H^0(BB_{a_1,b_1},\f{Z})\to H^4(BB_{a_1,b_1},\f{Z})\xrightarrow{B\pi^*}\zco{M_\du{a}{b}}{4}\to0.\end{equation}
 	\begin{lem}\label{orbcoho4}
		\(\zco{B_{a_1,b_1}}{4}\cong\f{Z}\).  If \(a_1\) and \(b_1\) are relatively prime then \(H^4(BB_{a_1,b_1},\f{Z})\cong\f{Z}\).
	\end{lem}

\begin{proof}
We consider the decomposition \eqref{bdecomp}. \(S^3/Q\) is an orientable 3-manifold.  \(S^3\times_{K_-}D^2\) deformation retracts onto \(S^3/K_-.\)  The action of  \(K_-\) on \(S^3\) has ineffective kernel \(\f{Z}_{|a_1|}\), and the action of \(K_-/\f{Z}_{|a_1|}\) is free with quotient \(\ar{}P^2.\)  	 The first statement of the lemma follows by applying the Meyer-Vietoris sequence to the decomposition.  
	 
	To prove the second statement, we examine the exact sequence  \eqref{gysin} to conclude that \(\zco{BB_{a_1,b_1}}4\cong\f{Z}\oplus T\), where \(T\) is a finite abelian group with order dividing \(n(\du{a}{b}).\)  Assume 3 does not divide \(b_1\).  Since \(T\) does not depend on \(a_2,a_3,b_2,b_3,\) we can apply this argument to triplets 
	 \[\underline{a}=(a_1,5,-3)\quad\text{and}\quad \underline{b}=(b_1,5,1)\]
	which satisfy \eqref{cond1} and \eqref{cond2} to conclude that \(|T|\) divides \(3a_1^2-2b_1^2\)
	 or to 
	 \[\underline{a}=(a_1,5,-3)\quad\text{and}\quad\underline{b}=(b_1,9,-7)\]
	 to conclude that \(|T|\) divides \(4a_1^2-2b_1^2.\)  Thus \(|T|\) divides \(a_1^2\) and  \(2b_1^2.\)  Since \(a_1\) and \(b_1\) are relatively prime and odd  \(|T|=1\).  If \(3\) divides \(b_1,\) we exchange  the values of \(a_2\) and \(b_2\), and \(a_3\) and \(b_3\), in the two pairs of triples used in the argument.

\end{proof}
\begin{lem}\label{pontlem}
	Let \(a_1\) and \(b_1\) be relatively prime.  Then 
	\[p_1(TM_\du{a}{b})=\pm2a_1^2m(\du{a}{b}){\bf 1}_\du{a}{b}.\]
\end{lem}

\begin{proof}
For ease of notation let \(n=n(\du{a}{b})\), \(m=m(\du{a}{b})\), and suppress the \(\du{a}{b}\) and \(a_1,b_1\) subscripts in the  manifolds and orbifolds described above.

Since \(M\) is a manifold, \(p_1(TM)\) and \(p_1^\text{orb}(TM)\) are identified by the homotopy equivalence \(M\simeq BM.\)  If \(Bi:BM\to BW\) is the inclusion, then \(p_1(TM)=Bi^*p_1^\text{orb}(TW).\)  Since \(TW\) is isomorphic to the sum of the pullbacks of \(TB\) and \(E\), and \(Bi\) and \(B\pi\) commute with the projection from \(BW\) to \(BB\), \(p_1(TM)=B\pi^*(p_1^\text{orb}(TB)+p_1^\text{orb}(E)).\)  Consider the commuting diagram 

\[
\begin{tikzcd}\zco{M}{4}\arrow{d}&\arrow{l}{\pi^*}\zco{B}{4}\arrow{r}\arrow{d}& H^4(B,\f{R})\arrow{d}\\
\zco{BM}{4}&\arrow{l}{B\pi^*}\zco{BB}{4}\arrow{r}& H^4(BB,\f{R})
\end{tikzcd}
.\]
 
	By \lref{orbcoho4} we can identify  the rightward horizontal maps with the inclusion \(\f{Z}\hookrightarrow\ar{}.\)  By \eqref{generalbaseint}   \(p_1(E)+p_1(TB)=-2m\in H^4(B,\ar{}).\) Recall from \sref{orbiclasses} that Chern-Weil and orbifold characteristic classes agree in \(H^4(BB,\ar{})\). Thus \(p_1^\text{orb}(TB)+p_1^\text{orb}(E)\) is the image in \(\zco{BB}{4}\) of \(-2m\in\zco{B}{4}.\)  We conclude that \(p_1(TM)=\pi^*(-2m).\)

	It remains to determine \(\pi^*\) in terms of the generator \({\bf1}\in\zco{M}{4}\) used to describe the linking form in \eqref{linkingform}.  In \cite[Section 3]{GKSl}, that generator is described using the decomposition \( M=M_-\cup_{M_0}M_+\), where 
	\[M_\pm=1\times \Delta S^3\bs G\times_{K_\pm}D^2\]
	\[M_0=1\times \Delta S^3\bs G/H.\] 
	
	Then \({\bf1}\) is the image of a generator of \(\zco{M,M_+}{4}\cong\f{Z}.\)  \ \(\pi\) maps the decomposition of \(M\) onto the decomposition \(B=B_-\cup_{B_0}B_+\) described in \eqref{bdecomp}.    Consider the commuting diagram
	
	\[\begin{tikzcd}
	\zco{M,M_+}{4}\arrow{r}&\zco{M}{4}\\\zco{B,B_+}{4}\arrow{u}\arrow{r}&\zco{B}{4}\arrow{u}{\pi^*}
	\end{tikzcd}.\]
  As seen in the proof of \lref{orbcoho4} \(B_+\) is homotopy equivalent to \(\ar{}P^2.\)  Thus	the long exact sequence of a pair implies that the bottom arrow is an isomorphism.  Identifying \(\zco{B,B_+}{4}\) and \(\zco{B}{4}\) with \(\f{Z}\) using that isomorphism and \lref{orbcoho4}, the left vertical map is multiplication by some integer \(k,\) and \(\pi^*(1)=k{\bf1}.\)

 	
 	To determine \(k\), we use excision to identify \(\zco{B,B_+}{4}\to\zco{M,M_+}{4}\) with the rightmost vertical map in the following commuting diagram with exact rows:
 	\[\begin{tikzcd}
 	\zco{M_-}{3}\arrow{r}&\zco{M_0}{3}\arrow{r}&\zco{M_-,M_0}{4}\\\zco{B_-}{3}\arrow{r}\arrow{u}&\zco{B_0}{3}\arrow{r}\arrow{u}&\zco{B_-,B_0}{4}\arrow{u}
 	\end{tikzcd}\]

	Since \(B_-\) is homotopy equivalent to \(\ar{}P^2\), the bottom right horizontal map is an isomorphism.  In the proof of \cite[Theorem 2.8]{GKS} the authors determine that \(\zco{M_0}{3}\) can be identified with \(\f{Z}^2\) in such a way that the image of the upper left horizontal map is generated by \((\ov{8}\left(a_2^2-a_3^2\right),a_1^2)\) and the image of the central vertical map by \((1,0).\)  It follows that the rightmost vertical map is multiplication by \(k=\pm a_1^2,\) depending on the choice of generators.    
	
	
We conclude that \(p_1(TM)=\pi^*(-2m)=\pm 2a_1^2m{\bf 1}\).  


	

\end{proof}

	\section{Definitions of signs in \(\hat{A}_\Lambda\)}\label{signs}
	In this section, we give a full definition of the of the signs \(\ep(g)\) in \eqref{ahatgdef} and \(\ep_i\) in \lref{lgm}, and derive the formulas for them in the relevant cases.  
	
	Let \(a\in\SO(2m)\) be a block diagonal matrix with \(m\) \(2\times 2\) blocks \[\mat{\cos\phi_l}{-\sin\phi_l}{\sin\phi_l}{\cos\phi_l},\ l=1,...,m.\]  If \(\{e_1,...,e_{2m}\}\) is the standard basis of \(\ar{2m}\) and we consider \(\Spin(2n)\) as a subset of the Clifford algebra generated by that basis, then the two preimages of \(a\) are \(\pm\tilde{a}\in\Spin(2m)\) where
	
	\[\tilde a=\prod_{l=1}^m\left(\cos\left(\frac{\phi_k}{2}\right)+\sin\left(\frac{\phi_k}{2}\right)e_{2l-1}e_{2l}\right).\]
	Note that replacing \(\phi_k\) with \(\phi_k+2\pi\) does not change \(a,\) bu flips the sign on \(\tilde a.\)
	
	Consider the setup used to define \(\hat{A}_g(h)\) in \sref{index}.  That is, let \(U_p\) be an orbifold chart with isotropy group \(\Gamma_p\) and \(\Gamma_p\) invariant metric \(h\) for a spin orbifold \(X\), and let \(g\in\Gamma_p.\)  Then there is an orthonormal basis \(\{e_1,...,e_{2m},e_{2m+1},...,e_n\}\) for \(T_pU_p\) such that \(e_1,...,e_{2m}\in N_pU_p^g\), \(e_{2m+1},...,e_n\in T_pU_p^g,\) and 
	\[g\in\SO(N_pU^g_p)\subset \SO(T_pU_p)\]
	is given in that basis by a matrix of the form \(a\in\SO(2m)\cong\SO(N_pU_p^g).\)  Each \(\phi_k\) will be equal to some \(\theta_j\), as defined in \sref{index}, but values of \(\phi_k\) may be repeated.  The spin structure provides a specific lift \(\tilde g\in\Spin(2m)\cong\Spin(N_pU_p^g)\subset \Spin(T_pU_p)\) which will be given in the basis by \(\pm\tilde a.\)  \(\ep(g)\) in \eqref{ahatgdef} is defined such that \(\tilde g=\ep(g)\tilde a.\)As mentioned above, the definition of \(\ep(g)\) depends both on the particular lift, and the choice of \(\theta_j\in\{\theta_j+2\pi\f{Z}\}\).
	
	Now assume further that \(g\) has odd order \(r\).  Then \(g^r=\text{id}\in\SO(T_pU_p),\) so \(r\phi_k\in2\pi\f{Z}.\)  Since the lift of \(\Gamma_p\) to \(\Spin(T_pU_p)\) is isomorphic, \(\tilde g\) also has order \(r,\) and so in the basis
	\[1=\tilde g^r=\ep(g)^r\tilde a^r=\ep(g)\prod_{l=1}^m\cos\left(\frac{r\phi_l}{2}\right)=\ep(g)\prod_{j,k}\cos\left(\frac{r\theta_j}{2}\right)\]
	where \(k\) indexes the formal splitting of \(N_j,\) accounting for the repeated values of \(\phi_l\).  
	
	Next, consider the setup of \sref{quotientinertia}.  If \(G\bs M\) is a spin orbifold, then the horizontal bundle \(H\subset TM\) of the \(G\) action on \(M\) admits a \(G-\) invariant spin structure.  Let \((g,p)\in\bar I_i.\)  Then \(g\in\SO(H_p)\) lifts to \(\tilde g\in\Spin(H_p).\)  Identifying \(TM\) as a subbundle of \(T(G\times M),\) at \((g,p)\)  \(dT|_{T_pM}=g\in\SO(T_pM).\)     Using the decomposition in \eqref{projcomp}, we can find a basis of \(H_p\) such that \(dT|_{H_p}=g|_{H_p}\in\SO(T\overline{N}_{(g,p)})\subset \SO(H_p)\) is given by a matrix of the form of \(a\in\SO(2m)\cong\SO(\overline{N}_{(g,p)})\) as above, with \(\phi_k=\theta_{i,j}\), as defined in \sref{quotientinertia} for some \(i,j\).  Thus \(\tilde g\in\Spin(2m)\cong\Spin(\overline{N}_{(g,p)})\) is either \(\pm\tilde a\) and \(\ep_i\) is defined such that \(\tilde g=\ep_i\tilde a.\)  By the same argument as give above, if the order of \(g\) is an odd number \(r,\)
	\[\ep_i=\prod_{j,k}\cos\left(\frac{r\theta_{i,j}}{2}\right)\]
	with \(k\) indexing the formal splitting of \(\overline{N}_{i,j}\)

	 Finally, we complete the proof of \lref{lgm} by showing that for \((g,p)\in\bar I_i, \ep(g)=\ep_i\).  In that proof, we define a slice neighborhood \(U_p\subset M\) as an orbifold chart for \(G\bs M,\) such that \(T_pU_p=H_p.\)  Then the lift of \(g\in SO(T_pU_p)\) to \(\Spin(T_pU_p)\) is given by the lift of the \(G\) action on \(\SO(H)\) to \(\Spin(H)\), and thus \(\ep(g)\) and \(\ep_i\) are defined by the same lift.  

\end{document}